\documentclass[11pt]{amsart}
\usepackage{mathtools}
\usepackage{graphicx}
\usepackage{color}
\usepackage[update,prepend]{epstopdf}
\usepackage{amsmath,amsfonts}
\usepackage{times}
\usepackage{cancel}
\usepackage[ansinew]{inputenc}
\usepackage{amssymb}
\usepackage{latexsym}
\usepackage{ulem}
\usepackage{cancel}

\numberwithin{equation}{section}

\newtheorem{corollary}{Corollary}[section]

\newtheorem{lemma}[corollary]{Lemma}

\newtheorem{remark}[corollary]{Remark}
\newtheorem{theorem}[corollary]{Theorem}

\begin{document}
\date{}
\thispagestyle{empty}
\title{Second sound thermoelastic stability of a string/beam structure}

\author{Farhat Shel}
\address{LR Analysis and Control of PDEs, LR 22ES03, Department of Mathematics, Faculty of Sciences of Monastir, University of Monastir, Tunisia}

\email{farhat.shel@fsm.rnu.tn}

\begin{abstract}

In this paper we study the one dimensional thermoelastic transmission problem in a special string/beam structure: the two components are coupled at an interface (identified to $0$). Either the string or the beam is supposed thermoelastic, the heat flux is given by the Cattaneo's law instead of the usual Fourier's law. We prove that the the energy decay of the whole system is exponential if the string is thermoelastic. When only the beam is thermoelastic, we prove that the energy of the coupling string/beam decays polynomially to zero as $\frac{1}{t}$ and the decay rate can be, at most, polynomially stable of order $\frac{1}{t^2}$. 
\end{abstract}


\maketitle

\tableofcontents

\section{Introduction}

The simplest classical equations of heat conduction are given by
\begin{eqnarray}
  \theta_t+\gamma \nabla q&=&0, \label{heat}\\
  q+\kappa\nabla \theta&=&0\;\;\;\;\;\text{(Fourier's law)} \label{four}.
\end{eqnarray}

with $\gamma, \kappa>0$ and where $\theta$ and $q$ are respectively, the temperature difference and the heat flux vector. Combining the previous two equations yields the parabolic Fourier equation of thermal conduction
$$\theta_t-\gamma \kappa \Delta \theta=0$$ 
that predicts an infinite speed of heat propagation \cite{Cat58, Ver58} and leads to an unphysical phenomena when the temporal scale goes down far from normal \cite{ Ver58}. To overcome this contradiction of Fourier's law (\ref{four})  we consider here the Cattaneo's law  of thermal conduction \cite{Cat58, Ver58}:
\begin{equation}
\tau q_t+q+\kappa\nabla\theta=0 \label{catt}
\end{equation}
with $\tau>0,$ to obtain, by combining (\ref{heat}) and (\ref{catt}), the hyperbolic damped wave equation
\begin{equation}
\tau \theta_{tt}+\theta_t-\Delta\theta=0.
\end{equation}
This phenomenon is known as second sound effect, and it is observed at a very low temperature.

Over the past decades, many authors studied the asymptotic stability of thermoelastic system with second sound \cite{MeSa05, Rac02, Rac03b}. Often, as in the case of Fourier's law for heat conduction, some results of exponential stability are proved. But it is not always the case for the two laws. In \cite{SaRa09} the authors proved that the dissipative effects of heat conduction induced by Cattaneo's law are weaker than those induced by Fourier's law, proving that the Cattaneo's law may destroy the exponential stability obtained via Fourier's law. see also \cite{SaRi12}.

Transmission problem in thermoelasticity with Fourier's law coupling an elastic part and a thermoelastic part of a material such as string or beam was studied since about two decades. In \cite{Mar02, Rac08} the authors considered the case of strings and proved that the presence of heat conduction in a part of the elastic system, induces exponential stability of the whole system as in the case of a single thermoelastic string. See also \cite{She13} for a general case of tree. In \cite{Riv01} 
the authors considered a transmission problem coupling a thermoelastic beam and a purely elastic one and proved that, as a single thermoelastic beam, the whole system is exponentially stable.

When replacing Fourier's law by Cattaneo's law for the heat conduction, Racke and \textit{al}.\cite{SRR08} obtained the same result for the transmission problem in elastic-thermoelastic stings. See also \cite{MeWa15} for a nonlinear case and \cite{RiWa20} for some recent remarks on thermoelastic transmission problem. In the case of a coupling of elastic-thermoelastic beams, naturally, we don't hope to obtain a polynomial decay rate of the associated energy
 batter than $\frac{1}{t^2}$ obtained for a single thermoelastic one.

In our previous paper   \cite{She20}, we considered a string-beam structure. We have proved that the exponential stability still holds if at least the string is thermoelastic (with the Fourier's law) and polynomially stable when only the beam is supposed thermoelastic (with the Fourier's law) . 

In the present paper we deal with a string-beam coupling, one of the two component is damped by a heat conduction induced by Cattaneo's law. So we consider two PDEs' systems, the first represent a string $e_1$ of length $\ell_1$ and the second represent a beam $e_2$ of length $\ell_2$, coupled at the interface point $0.$ The heat conduction
 is coupled either with the wave equation or with the Petrowski equation. More precisely we consider the two initial boundary value problems: The first is described by equations (\ref{sc-b})-(\ref{btci}) below
\begin{equation}
\left\{
   \begin{aligned}
  u_{1,tt}-\alpha _{1}u_{1,xx}+\beta _{1}\theta _{1,x} &=&0\text{ in }(0,\ell
_{1})\times (0,\infty ),  \\
           \theta _{1,t}+\gamma _{1}q_{1,x}+\delta _{1}u_{1,tx} &=&0\text{ in }(0,\ell
_{1})\times (0,\infty ),  \\
\tau _{1}q_{1,t}+q_{1}+\kappa _{1}\theta _{1,x} &=&0\text{ in }(0,\ell
_{1})\times (0,\infty ),  \\           
          u_{2,tt}+\alpha _{2}u_{2,xxxx}&=&0\text{ in }(0,\ell _{2})\times (0,\infty ), 
           \end{aligned}
           \right. \label{sc-b}
\end{equation}
\begin{equation}
\left\{ 
\begin{tabular}{l}
$u_{1}(0,t)=u_{2}(0,t),\;\;\;\theta _{1}(0,t)=0,\;\;\;u_{2,x}(0,t)=0,$ \\ 
$ \alpha _{2}u_{2,xxx}(0,t) =\frac{\delta _{1}}{\beta _{1}} \alpha
_{1}u_{1,x}(0,t),$\\
$ u_{1}(\ell _{1},t)=0,\;\;\theta _{1}(\ell _{1},t)=0,u_{2}(\ell
_{2},t)=0,\;\;u_{2,xx}(\ell _{2},t)=0,$
\label{b2}
\end{tabular}
\right.
\end{equation}
\begin{equation}
\left\{ 
\begin{tabular}{l}
$
u_{j}(x,0)=u_{j}^0(x),\;u_{j,t}(x,0)=u_{j}^1(x),\;j=1,2,$\\
$\theta _{1}(x,0)=\theta
_{1}^0(x),\;q_{1}(x,0)=q_{1}^0(x).$  \label{btci}
\end{tabular}
\right.
\end{equation}
The second is described by equations (\ref{s-bc})-(\ref{btcii}) below
\begin{equation}
\left\{
   \begin{aligned}
   u_{1,tt}-\alpha _{1}u_{1,xx}&=&0\text{ in }(0,\ell _{1})\times (0,\infty ),\\
 u_{2,tt}+\alpha _{2}u_{2,xxxx}-\beta _{2}\theta _{2,xx} &=&0\text{ in }%
(0,\ell _{2})\times (0,\infty ), \\
\theta _{2,t}+\gamma _{2}q_{2,x}+\delta _{2}u_{2,txx} &=&0\text{ in }(0,\ell
_{2})\times (0,\infty ), \\
\tau _{2}q_{2,t}+q_{2}+\kappa _{2}\theta _{2,x} &=&0\text{ in }(0,\ell
_{2})\times (0,\infty ),\\ 
           \end{aligned}
           \right. \label{s-bc}
\end{equation}
\begin{equation}
\left\{ 
\begin{tabular}{l}
$u_{1}(0,t)=u_{2}(0,t),\;\;\;\theta _{2}(0,t)=0,\;\;\;u_{2,x}(0,t)=0,$ \\ 
$\frac{\delta _{2}}{\beta _{2}}\left( \alpha _{2}u_{2,xxx}(0,t)-\beta
_{2}\theta _{2,x}(0,t)\right) = \alpha
_{1}u_{1,x}(0,t),$\\
$ u_{1}(\ell _{1},t)=0,u_{2}(\ell
_{2},t)=0,\;\;\theta _{2}(\ell _{2},t)=0,\;\;u_{2,xx}(\ell _{2},t)=0,$
\label{b22}
\end{tabular}
\right.
\end{equation}
\begin{equation}
\left\{ 
\begin{tabular}{l}
$
u_{j}(x,0)=u_{j}^0(x),\;u_{j,t}(x,0)=u_{j}^1(x),\;j=1,2,$\\
$\theta _{2}(x,0)=\theta
_{2}^0(x),\;q_{2}(x,0)=q_{2}^0(x),$  \label{btcii}
\end{tabular}
\right.
\end{equation}
with positive constants $\alpha_j, \beta_j, \gamma_j, \delta_j, \tau_j, \kappa_j,\;\;j=1,2.$
The functions $
u_{j}=u_{j}(x,t),$ $\theta _{j}=\theta _{j}(x,t)$ and $q_{j}=q_{j}(x,t)$
 represent the displacement, the temperature
difference to a fixed reference temperature, and the heat flux  at time $t \in [0,\infty)$ and position $x \in(0,\ell_j)$, respectively,
 for $j=1,2.$ 

Note that that the cases $\tau_1=0$ and $\tau_2=0,$ that is assuming Fourier's law, correspond to the second and third systems studied in \cite{She20}.

The energy functions for the systems above are given respectively by
\begin{eqnarray*}
E_1(t) &=&\frac{1}{2}\int_{0}^{\ell _{1}}\left( \frac{\delta _{1}}{\beta _{1}}%
\left| u_{1,t}\right| ^{2}+\frac{\delta _{1}\alpha _{1}}{\beta _{1}}\left|
u_{1,x}\right| ^{2}+\left| \theta _{1}\right| ^{2}+\frac{\gamma _{1}\tau _{1}%
}{\kappa _{1}}\left| q_{1}\right| ^{2}\right) dx \\
&&+\frac{1}{2}\int_{0}^{\ell _{2}}\left( 
\left| u_{2,t}\right| ^{2}+\alpha _{2}\left|
u_{2,xx}\right| ^{2}\right) dx,\;\;\;\forall\,t\geq 0
\end{eqnarray*}
and
\begin{eqnarray*}
E_2(t) &=&\frac{1}{2}\int_{0}^{\ell _{1}}\left( 
\left| u_{1,t}\right| ^{2}+\alpha _{1}\left|
u_{1,x}\right| ^{2}\right) dx \\
&&+\frac{1}{2}\int_{0}^{\ell _{2}}\left( \frac{\delta _{2}}{\beta _{2}}%
\left| u_{2,t}\right| ^{2}+\frac{\delta _{2}\alpha _{2}}{\beta _{2}}\left|
u_{2,xx}\right| ^{2}+\left| \theta _{2}\right| ^{2}+\frac{\gamma _{2}\tau
_{2}}{\kappa _{2}}\left| q_{2}\right| ^{2}\right) dx,\;\;\;\forall\,t\geq 0.
\end{eqnarray*}
Their derivatives with respect to time $t,$ are given by 
\begin{equation*}
\frac{d}{dt}E_j(t) =-\frac{\gamma _{j}}{\kappa _{j}}\left\| q_{j}\right\| ^{2},\;\;\;j=1,2.
\end{equation*}
So the energies above are non-increasing. In this work we study the large-time behavior of the solution in the two cases. We prove that if the heat is effective in the string part, then the string-beam system is exponentially stable, while if the heat is only effective on the beam part, we don't hope obtain an exponential stability of the whole system, because we already know that a thermoelastic beam has an optimal (polynomial) decay rate of order $\frac{1}{t^2}$ and a purely elastic string is conservative. We will prove that the energy of the coupling string/thermoelastic beam decay polynomially to zero as $\frac{1}{t}$ and cannot exceed the order $2$.  

The remaining part of this paper is organized as follows. In section 2 we shall look at the coupling thermoelastic string-purely elastic beam described by (\ref{sc-b})-(\ref{btci}) and prove that it is exponentially stable. Section 3 is devoted to discussing the large time behavior of the coupling string-beam where only the beam is thermoelastic, described by (\ref{s-bc})-(\ref{btcii}). We show that the system is polynomially stable and non exponentially stable.

\section{String Cattaneo/Beam}

\subsection{Abstract setting}

First we will write the system (\ref{sc-b})-(\ref{btci}) as an evolution equation into an appropriate
Hilbert space, we establish the well-posedness of the system, then we estimate the stability of the corresponding semigroup.

Let us consider 
\begin{equation*}
V=\left\{ f=(f_{1},f_{2})\in H^{1}(0,\ell _{1})\times H^{2}(0,\ell _{2})\mid
f\text{ satisfies (\ref{s1'})}\right\}
\end{equation*}
where 
\begin{equation}
\left\{ 
\begin{tabular}{l}
$f_{1}(\ell _{1})=0,\;\;f_{2}(\ell _{2})=0,$ \\ 
$f_{1}(0)=f_{2}(0),$ \\ 
$\partial _{x}f_{2}(\ell _{2})=0.$%
\end{tabular}
\right.  \label{s1'}
\end{equation}

\noindent Define the Hilbert space $\mathcal{H}_{1}$ 
\begin{equation*}
\mathcal{H}_{1}=V\times \left(L^{2}(0,\ell_1)\times L^{2}(0,\ell_2)\right) \times L^{2}(0,\ell_1)\times L^{2}(0,\ell_1)
\end{equation*}
with norm given by 
\begin{equation*}
\left\| y\right\| _{\mathcal{H}_1}^2:=\frac{\delta _{1}\alpha _{1}}{\beta _{1}}%
\left\| \partial _{x}f_{1}\right\| ^{2}+\frac{\delta _{1}}{\beta
_{1}}\left\| g_{1}\right\| ^{2}+\left\| h_{1}\right\| ^{2}+\frac{\gamma
_{1}\tau _{1}}{\kappa _{1}}\left\| d_{1}\right\| ^{2}+\alpha_2\left\| \partial
_{x}^{2}f_{2}\right\| ^{2}+\left\| g_{1}\right\| ^{2},
\end{equation*}
where $z=\left( f,g,h_{1},d_{1}\right) $.

Now define the unbounded operator $\mathcal{A}_{1}$ on $\mathcal{H}_{1}$ by 
\begin{equation*}
\mathcal{D}(\mathcal{A}_{1})=\left\{ 
\begin{array}{c}
y=(u,v,\theta _{1},q_{1})\in V\cap (H^{2}(0,\ell _{1})\times H^{4}(0,\ell
_{2}))\times V\times H^{1}_0(0,\ell _{1})\times H^{1}(0,\ell _{1})\mid \\ 
\text{ and }y\text{ satisfies (\ref{s2''})}
\end{array}
\right\}
\end{equation*}
where 
\begin{equation}
\left\{ 
\begin{tabular}{l}
$\partial _{x}^{2}u_{2}(\ell _{2})=0,$ \\ 
$\alpha _{2}\partial _{x}^{3}u_{2}(0)=\alpha _{1}\partial _{x}u_{1}(0),$ 
\end{tabular}
\right. \label{s2''}
\end{equation}
and
\begin{equation*}
\mathcal{A}_{1}\left( 
\begin{array}{c}
u_{1} \\ 
u_{2} \\ 
v_{1} \\ 
v_{2} \\ 
\theta _{1} \\ 
q_{1}
\end{array}
\right) =\left( 
\begin{array}{c}
v_{1} \\ 
v_{2} \\ 
\alpha _{1}\partial _{x}^{2}u_{1}-\beta _{1}\partial _{x}\theta _{1} \\ 
-\alpha _{2}\partial _{x}^{4}u_{2} \\ 
-\gamma _{1}\partial _{x}v_{1}-\delta _{1}\partial _{x}q_{1} \\ 
-\frac{1}{\tau _{1}}q_{1}-\frac{\kappa _{1}}{\tau _{1}}\partial _{x}\theta
_{1}
\end{array}
\right) .
\end{equation*}

Setting  $y(t)=(u(.,t),u_t(.,t),\theta _{1}(.,t),q_{1}(.,t))$ and 
 $y^{0}=(u^{0},v^{0},\theta
_{1}^0,q_{1}^0)\in\mathcal{H}_1$, 

The system (\ref{sc-b})-(\ref{btci}) my be recast in $\mathcal{H}_1$ as
\begin{equation}
\left\{ 
\begin{array}{c}
\frac{dy}{dt}(t)=\mathcal{A}_{1}y(t),\;\;\;t>0, \\ 
y(0)=y^{0}.
\end{array}
\right.  \label{0500}
\end{equation}

We have the following result, 
\begin{theorem}[Well-posedness] \label{lem1}
The operator $\mathcal{A}_1$ generates a $\mathcal{C}_0$-semigroup of contractions $(S_1(t))_{t\geq 0}$, which is strongly stable:
$$\lim_{t\rightarrow \infty}\|S(t)y^0\|_{\mathcal{H}_1}=0,\;\;\;\forall\,y^0\in\mathcal{H}_1.$$
\end{theorem}

\begin{proof}
The proof of the well-posedness is based on the Lumer-Philips theorem found e.g. in \cite{Paz83}; il suffices to prove that  $\mathcal{A}_1$ is dissipative and $0\in\rho(\mathcal{A}_1)$ the resolvent set of $\mathcal{A}_1$. The dissipativeness of $\mathcal{A}_1)$ is immediate, since $$Re(\left\langle \mathcal{A}_{1}y,y\right\rangle _{\mathcal{H}_{1}})=-\frac{%
\gamma _{1}}{\kappa _{1}}\left\| q_{1}\right\| ^{2}\leq 0,\;\;\;y=(u,v,\theta _{1},q_{1})\in \mathcal{D}(\mathcal{A}_1).$$  
Now, for a fixed $F=(f,g,h,d) \in \mathcal{H}_1$, let us consider the solvability of equation $\mathcal{A}_1y=F$, $y\in\mathcal{D}(\mathcal{A}_1)$, that is,
\begin{equation}
\left\{ 
\begin{tabular}{l}
$v_{1}=f_{1},$ \\
$\alpha _{1}\partial _{x}^{2}u_{1}-\beta _{1}\partial _{x}\theta _{1}=g_{1},$\\
$-\gamma _{1}\partial _{x}q_{1}=\delta _{1}\partial _{x}f_{1}+h_{1},$ \\
$-q_{1}-\kappa _{1}\partial _{x}\theta _{1}=\tau _{1}d_{1},$

\end{tabular}
\right.  \label{21}
\end{equation}
\begin{equation}
\left\{ 
\begin{tabular}{l}
$v_{2}=f_{2},$ \\
$-\alpha _{2}\partial _{x}^{4}u_{2}=g_{2},$%
\end{tabular}
\right.  \label{22}
\end{equation}
with boundary and transmission conditions
\begin{equation}
\theta _{1}(\ell _{1})=0,\;\;\theta_{1}(0)=0,\label{22'}
\end{equation}
\begin{equation}
u_{1}(\ell _{1})=0,\;\;\;u_{2}(\ell _{2})=0,\;\;\;\partial _{x}^{2}u_{2}(\ell _{2})=0,
\label{23}
\end{equation}
and
\begin{equation}
u_{1}(0)=u_{2}(0),\;\;\;\partial _{x}u_{2}(0)=0,\;\;\;\alpha _{2}\partial _{x}^{3}u_{2}(0)=\alpha _{1}\partial _{x}u_{1}(0).
\label{24}
\end{equation}

Integrating the third equation in (\ref{21}) from $0$ to $x$, we obtain
\begin{equation}
-\gamma _{1}q_{1}=\delta _{1}f_{1}+\int_{0}^{x}h_{1}(s)ds+
a_{1}.\label{25}
\end{equation}
Substituting (\ref{25}) into the fourth equation in (\ref{21}), then integrating by part the obtained identity from $0$ to $x$ yields
\begin{equation}
\gamma _{1}\kappa _{1}\theta _{1} =\underset{:=F_{1}(x)}{\underbrace{%
\int_{0}^{x}(\delta _{1}f_{1}(s)-\gamma _{1}\tau
_{1}d_{1}(s))ds+\int_{0}^{x}(x-s)h_{1}(s)ds}}+a_{1}x+b_{1},
\end{equation}
where $a_1,b_1$ are two constants.

Using that $\theta_1(0)=0$ and $\theta_1(\ell_1)=0$, we have $$b_1=0,\;\;\;a_1=-\frac{1}{\ell_1} F_1(\ell_1).$$
 Now, integrating the second equation in (\ref{21}) and the second
equations in (\ref{22}) several times, we get 
\begin{eqnarray*}
\alpha _{1}\partial _{x}u_{1}-\beta _{1}\theta _{1}
&=&\int_{0}^{x}g_{1}(s)ds+c_{1}, \\
\gamma _{1}\kappa _{1}\alpha _{1}u_{1} &=&\beta
_{1}\int_{0}^{x}F_{1}(s)ds+\gamma _{1}\kappa
_{1}\int_{0}^{x}(x-s)g_{1}(s))ds-\frac{\beta _{1}}{\ell_1}F_1(\ell_1)x^2
+\gamma _{1}\kappa _{1}c_{1}x+d_{1},
\end{eqnarray*}
and 
\begin{eqnarray*}
\alpha _{2}\partial _{x}^{3}u_{2}
&=&-\int_{0}^{x}g_{2}(s)ds+c_{2}, \\
\alpha _{2}\partial _{x}^{2}u_{2} &=&
-\int_{0}^{x}(x-s)g_{2}(s)ds+c_{2}x+d_2, \\
\alpha _{2}\partial _{x}u_{2}
&=&-\int_{0}^{x}\frac{(x-s)^2}{2} g_{2}(s)ds+\frac{1}{2}c_{2} x^{2}+d_{2}x+n_{2}, \\
\alpha _{2}u_{2} &=&-\int_{0}^{x}\frac{(x-s)^3}{3!} g_{2}(s)ds+\frac{1%
}{3!}c_{2}x^{3}+\frac{1%
}{2}d_{2}x^{2}+n_{2}x+m_{2}
\end{eqnarray*}
where $c_{1},$ $d_{1},$ $c_{2},$ $d_{2},n_{2}$ and $m_{2}$ are constants. 

The boundary and transmission conditions (\ref{23}) and (\ref{24}) are expressed as follow
\begin{eqnarray}
-\beta _{1}F_1(\ell_1)\ell _{1}
+\gamma _{1}\kappa _{1}c_{1}\ell _{1}+d_{1} &=&\Gamma_1 := -\beta
_{1}\int_{0}^{\ell_1}F_{1}(s)ds-\gamma _{1}\kappa
_{1}\int_{0}^{\ell_1}(x-s)g_{1}(s))ds  \label{s1} \\
\frac{1}{3!}c_{2} \ell
_{2}^{3}+\frac{1}{2}d_{2}\ell _{2}^{2}+n_{2}\ell _{2}+m_{2} &=&\Gamma_2:=\int_{0}^{\ell_2}\frac{(x-s)^3}{3!} g_{2}(s)ds,
\label{s2} \\
c_{2}\ell _{2}+d_{2} &=&\Gamma_3:= \int_{0}^{\ell_2}(x-s) g_{2}(s)ds,
\label{s3} \\
\alpha _{2}d_{1} &=&\gamma _{1}\kappa _{1}\alpha
_{1}m_{2}  \label{s4} \\
n_{2} &=&0  \label{s5} \\
c_{1} &=&c_{2},  \label{s6}
\end{eqnarray}
First, using (\ref{s6}), (\ref{s5}), (\ref{s4}) and (\ref{s3}) in (\ref{s2}) allows us to write $c_1$ in terms of $d_1$. Then from (\ref{s1}), $d_1$ is uniquely determined and consequently $c_1$ also. Finally, $c_2$, $d_2$ and $m_2$ are uniquely determined by (\ref{s6}), (\ref{s3}) and (\ref{s4}) respectively.

We conclude that the equation $\mathcal{A}_1y=F$ has a unique solution $y\in \mathcal{D}(\mathcal{A}_1).$ 
Moreover, from the regularity theory of linear elliptic systems \cite{Liu99}, there
exists a positive constant $C$ independent from $y$ and $z,$ such that 
\begin{equation*}
\left\| y\right\| _{\mathcal{H}_1}\leq C\left\| z\right\| _{\mathcal{H}_1}.
\end{equation*}
We conclude that $\mathcal{A}_1^{-1}\in \mathcal{L}(\mathcal{H}_1),$ that is $%
0\in \rho (\mathcal{A}_1).$
\end{proof}
\subsection{Strong stability}
We begin by the following lemma
\begin{lemma}\label{lem1'}
The operator $\mathcal{A}_1$ satisfies
\begin{equation}
\mathbf{i}\mathbb{R}\subset \rho (\mathcal{A}_1).  \label{2.1}
\end{equation}
\end{lemma}
\begin{proof}
First, by the Sobolev embedding theorem, $\mathcal{A}_1^{-1}$ is a compact
operator and then the spectrum of $\mathcal{A}_1$ consists of all isolated
eigenvalues, i.e. $\sigma (\mathcal{A}_1)=\sigma _{p}(\mathcal{A}_1)$, hence, we will prove that $\sigma _{p}(\mathcal{A}_1)\cap \mathbf{i}\mathbb{R}=\emptyset$. Otherwise,  there is a real number $\beta \neq 0,$ such that $%
\lambda :=i\beta $ is an eigenvalue of $\mathcal{A}_1.$ Let $y=(u,v,\theta_1 ,q_1)$
the corresponding eigenvector. We have 
\begin{equation}
\left\{ 
\begin{tabular}{l}
$
\begin{tabular}{l}
$v_{j}=\lambda u_{j}$ \ \ \ \ \ \ for $j$ in $\{1,2\},$ \\ 
$\alpha _{1}\partial _{x}^{2}u_{1}-\beta _{1}\partial _{x}\theta
_{1}=\lambda v_{1},$\\ $-\alpha _{1}\partial _{x}^{4}u_{2}=\lambda v_{2},$ \\ 
$-\delta _{1}\partial _{x}v_{1}-\gamma _{1}\partial
_{x}q_{1}=\lambda \theta _{1},$ \\ 
$-\frac{1}{\tau _{1}}q_{1}-\frac{\kappa _{1}}{\tau _{1}}\partial _{x}\theta
_{1}=\lambda q_{1}.$%
\end{tabular}
$%
\end{tabular}
\right.  \label{8.2}
\end{equation}

Taking the real part of the inner product of $\lambda y-\mathcal{A}_1%
y=0 $ with $y$ in $\mathcal{H}_1,$ we obtain 
\begin{equation*}
Re(\left\langle \mathcal{A}_1y,y\right\rangle _{\mathcal{H}_1})=-\frac{\gamma
_{1}}{\kappa _{1}}\left\| q_{1}\right\| ^{2}.
\end{equation*}

Thus $q_{1}=0$. Then, taking into account that $\theta_1,v_1,u_1 \in H^1_0(0,\ell_1)$, and using the fifth, fourth and first equations in (\ref{8.2}), it is easy to deduce that 
 $\theta_1=0$, $v_{1}=0$ and  $u_{1}=0$ on $[0,\ell_1]$. Immediately, by transmission and boundary conditions, we deduce that $\partial_x^ku_2(0)=0$ for $k=0,1,3$ and  $\partial_x^ku_2(\ell_2)=0$ for $k=0,2$. Hence, according to the third equation in (\ref{8.2}), we obtain that $u_2=0$ on $[0,\ell_2]$. 
\end{proof}

Thanks to  the strong stability criterion established in \cite{ArBa88} we have
\begin{theorem}[Strong stability]
The $\mathcal{C}_0$-semigroup $(S_1(t))_{t\geq 0}$ is strongly stable:
$$\lim_{t\rightarrow \infty}\|S(t)y^0\|_{\mathcal{H}_1}=0,\;\;\;\forall\,y^0\in\mathcal{H}_1.$$
\end{theorem}

\subsection{Exponential stability}

The aim of this section is to quantify the strong stability of $(S_1(t))_{t\geq 0}$ by showing that it is exponentially stable.
\begin{theorem}[Exponential stability]
The semigroup  $S_{1}(t)_{t\geq 0}$ is exponentially stable, i.e. there exists two positive constants $M$ and $\lambda$ such that
$$\|S_1(t)y^0\|_{\mathcal{H}_1}\leq Me^{-\lambda t}\|y^0\|_{\mathcal{H}_1},\;\;\;\forall\,y^0\in\mathcal{H}_1.$$
\end{theorem}
\begin{proof}
We will use the characterization of the exponential stability of semigroups given in \cite{Gea78}, \cite{Hua85} or \cite{Pru84} (see also \cite{Liu99}). Precisely, following Theorem 1.3.2 in \cite{Liu99}, taking into account that $\mathbf{i}\mathbb{R}\subset \rho (\mathcal{A}_1)$ (Lemma \ref{lem1'}), it suffices to show the following resolvent estimate
\begin{equation}
\underset{\left| \beta \right| \rightarrow \infty }{\lim }\sup \left\| (%
\mathbf{i}\beta -\mathcal{B})^{-1}\right\| <\infty.  \label{2.2}
\end{equation}
Suppose the conclusion is
false. Then there exists a sequence $(w_{n})$ of real numbers, with $%
|w_{n}|\longrightarrow +\infty $ (without loss of generality, we suppose that $w_n>0$), and a sequence of vectors $%
(y_{n})=(u_{n},v_{n},\theta _{n},q_{n})$ in $\mathcal{D}(\mathcal{A}_1)$ with $%
\left\| y_{n}\right\| _{\mathcal{H}_1}=1$, such that 
\begin{equation*}
\left\| (\mathbf{i}w_{n}I-\mathcal{A}_1)y_{n}\right\| _{\mathcal{H}_1%
}\longrightarrow 0
\end{equation*}
which is equivalent to 
\begin{eqnarray}
\mathbf{i}w_{n}u_{1,n}-v_{1,n} &=&f_{1,n}\longrightarrow 0,\;\;\;\text{in}%
\;H^{1}(0,\ell _{1}),  \label{e11} \\
\mathbf{i}w_{n}v_{1,n}-\alpha _{1}\partial _{x}^{2}u_{1,n}+\beta
_{1}\partial _{x}\theta _{1,n} &=&g_{1,n}\longrightarrow 0,\;\;\;\text{in}%
\;L^{2}(0,\ell _{1}),  \label{e12} \\
\mathbf{i}w_{n}\theta _{1,n}+\delta _{1}\partial _{x}v_{1,n}+\gamma
_{1}\partial _{x}q_{1,n} &=&h_{1,n}\longrightarrow 0,\;\;\;\text{in}%
\;L^{2}(0,\ell _{1}),  \label{e13} \\
\mathbf{i}w_{n}\tau _{1}q_{1,n}+q_{1,n}+\kappa _{1}\partial _{x}\theta
_{1,n} &=&d_{1,n}\longrightarrow 0,\;\;\;\text{in}\;L^{2}(0,\ell _{1}),
\label{e14}
\end{eqnarray}
and 
\begin{eqnarray}
\mathbf{i}w_{2,n}u_{2,n}-v_{2,n} &=&f_{2,n}\longrightarrow 0,\;\;\;\text{in}%
\;H^{2}(0,\ell _{2}),  \label{e21} \\
\mathbf{i}w_{n}v_{2,n}+\alpha _{2}\partial _{x}^{4}u_{2,n}
&=&g_{2,n}\longrightarrow 0,\;\;\;\text{in}\;L^{2}(0,\ell _{2}).  \label{e24}
\end{eqnarray}
Then 
\begin{eqnarray}
w_{n}^{2}u_{1,n}+\alpha _{1}\partial _{x}^{2}u_{1,n}-\beta _{1}\partial
_{x}\theta _{1,n} &=&-g_{1,n}-\mathbf{i}w_{n}f_{1,n},  \label{e15} \\
\theta _{1,n}+\delta _{1}\partial _{x}u_{1,n}+\gamma
_{1}\frac{1}{\mathbf{i}w_{n}}\partial _{x}q_{1,n} &=&\frac{1}{\mathbf{i}w_{n}}(h_{1,n}+\delta _{1}f_{1,n}),  \label{e13'} \\
-w_{n}^{2}u_{2,n}+\alpha _{2}\partial _{x}^{4}u_{2,n} &=&g_{2,n}+\mathbf{i}%
w_{n}f_{2,n},  \label{e35}
\end{eqnarray}
and 
\begin{equation}  \label{e36}
\left\| v_{j,n}\right\| ^{2}-w_{n}^{2}\left\| u_{j,n}\right\|
^{2}\longrightarrow 0,\;\;j=1,2.
\end{equation}

First, since 
\begin{equation*}
Re(\left\langle (\mathbf{i}w_{n}-\mathcal{A}_1)y_{n},y_{n}\right\rangle _{%
\mathcal{H}_1})=-\frac{\gamma _{1}}{\kappa _{1}}\left\| q_{1,n}\right\| ^{2},
\end{equation*}
we obtain that $q_{1,n}$ converges to $0$ in $%
L^{2}(0,\ell _{1})$ and then by (\ref{e14}) we
deduce that $\frac{\partial _{x}\theta _{1,n}}{w_{n}}$  converges to zero in $L^{2}(0,\ell
_{1})$.

Now, multiplying (\ref{e13'}) by $\theta _{1,n}$ leads to 
\begin{equation*}
\left\| \theta _{1,n}\right\| ^{2}+\left. \frac{\gamma _{1}}{\mathbf{i}w_{n}}q_{1,n}%
\overline{\theta _{1,n}}\right| _{0}^{\ell _{1}}-\frac{\gamma _{1}}{\mathbf{i}w_{n}}%
\left\langle q_{1,n},\partial _{x}\theta _{1,n}\right\rangle +\delta
_{1}\left. u_{1,n}\overline{\theta _{1,n}}\right| _{0}^{\ell _{1}}-\delta
_{1}\left\langle u_{1,n},\partial _{x}\theta _{1,n}\right\rangle =o(1).
\end{equation*}

Taking into account that $\frac{\partial _{x}\theta _{1,n}}{w_{n}}$ and $q_{1,n}$ converge
to $0$,  that $w_{n}u_{1,n}$ is bounded (by \ref{e11}), and that $\theta_1(\ell_1)=\theta_1(0)=0$, we have 
\begin{equation*}
\left\| \theta _{1,n}\right\| ^{2}=o(1).
\end{equation*}

Then, (\ref{e13'}) becomes 
\begin{equation}
\delta _{1}\partial _{x}u_{1,n}+\gamma _{1}\frac{\partial _{x}q_{1,n}}{iw_{n}%
}=o(1).  \label{e1555}
\end{equation}
Multiplying (\ref{e1555}) by $\partial _{x}u_{1,n}$ and integrating by
parts, it follows 
\begin{equation*}
\delta _{1}\left\| \partial _{x}u_{1,n}\right\| ^{2}-\frac{\gamma _{1}}{%
iw_{n}}\left\langle q_{1,n},\partial _{x}^{2}u_{1,n}\right\rangle +\left. 
\frac{\gamma _{1}}{iw_{n}}q_{1,n}\partial _{x}\overline{u_{1,n}}\right|
_{0}^{\ell _{1}}\longrightarrow 0.
\end{equation*}
Since $\frac{\partial _{x}^{2}u_{1,n}}{w_{n}}$ is bounded (by equation (\ref{e15})) and $q_{1,n}\longrightarrow 0$ in $L^{2}(0,\ell _{1})$, one gets 
\begin{equation*}
\delta _{1}\left\| \partial _{x}u_{1,n}\right\| ^{2}+\left. \frac{\gamma _{1}%
}{iw_{n}}q_{1,n}\partial _{x}\overline{u_{1,n}}\right| _{0}^{\ell
_{1}}\longrightarrow 0.
\end{equation*}
Using Gagliardo-Nirenberg inequality \cite{Liu99}, taking into account the boundedness of $\frac{\partial _{x}^{2}u_{1,n}}{w_{n}}$, the boundedness of $\frac{1}{\mathbf{i}w_{n}}\partial _{x}q_{1,n}$ (by (\ref{e13'})),   we have
$$\frac{\left\| \partial _{x}u_{1,n}\right\| _{\infty }}{\sqrt{w_{n}}}\leq C\left( \left\| \partial _{x}u_{1,n}\right\|^{1/2}\frac{\left\| \partial _{x}^2u_{1,n}\right\|^{1/2}}{\sqrt{w_{n}}}+\frac{\left\| \partial _{x}u_{1,n}\right\|}{\sqrt{w_{n}}}\right)=O(1) $$  and
$$\frac{\left\| q_{1,n}\right\| _{\infty }}{\sqrt{w_{n}}}\leq C\left( \left\| q_{1,n}\right\|^{1/2}\frac{\left\| \partial _{x}q_{1,n}\right\|^{1/2}}{\sqrt{w_{n}}}+\frac{\left\| q_{1,n}\right\|}{\sqrt{w_{n}}}\right)=o(1) $$
for some positive constant $C$.\\
It yields, 
\begin{equation*}
\left\| \partial _{x}u_{1,n}\right\| ^{2}\longrightarrow 0,\;\;\text{in }%
L^{2}(0,\ell _{1}).
\end{equation*}
Hence, by (\ref{e11}), $\frac{\partial _{x}v_{1,n}}{w_{n}}\longrightarrow
0,\;$in $L^{2}(0,\ell _{1}).$ Thus, multiplying (\ref{e12}) by $\frac{1}{%
\mathbf{i}w_{n}}v_{1,n}$ to get 
\begin{equation*}
\left\| v_{1,n}\right\| ^{2}-\frac{\alpha _{1}}{iw_{n}}\left. \partial
_{x}u_{1,n}\overline{v_{1,n}}\right| _{0}^{\ell _{1}}\longrightarrow 0.
\end{equation*}
On one hand, recall that $\frac{\left\| \partial
_{x}u_{1,n}\right\| _{\infty }}{\sqrt{w_{n}}}$ converges to zero. On the other hand, to estimate $\frac{\left\| v_{1,n}\right\| _{\infty }}{%
\sqrt{w_{n}}}$  we apply again the Gagliardo-Nirenberg inequality,
$$\frac{\left\| v_{1,n}\right\| _{\infty }}{\sqrt{w_{n}}}\leq C\left( \left\| v_{1,n}\right\|^{1/2}\frac{\left\| \partial _{x}v_{1,n}\right\|^{1/2}}{\sqrt{w_{n}}}+\frac{\left\| v_{1,n}\right\|}{\sqrt{w_{n}}}\right)=o(1). $$
Then
\begin{equation*}
\left\| v_{1,n}\right\| ^{2}\longrightarrow 0.
\end{equation*}

In conclusion 
\begin{equation}
\left\| w_{n}u_{1,n}\right\| ,\;\left\| v_{1,n}\right\| ,\;\left\| \partial
_{x}u_{1,n}\right\| ,\;\left\| \theta _{1,n}\right\| \longrightarrow 0. \label{ee0}
\end{equation}

Now, let $p_1(x)=\ell_1-x$ and taking the real part of the inner product of (\ref{e15}), (\ref{e13'}) and (\ref{e14}) by $p_1\partial _{x}u_{1,n},$  $p_1\partial _{x}\theta_{1,n}$ and $p_1\frac{1}{\mathbf{i}w_n}\partial _{x}q_{1,n}$ respectively (in $L^2(0,\ell_1),$  we obtain (taking in mind boundary conditions),
\begin{eqnarray}
-\frac{\ell_1}{2}w_{n}^{2}\left| u_{1,n}(0)\right|
^{2}+\frac{1}{2}\int_{0}^{\ell _{1}}w_{n}^{2}\left| u_{1,n}\right|
^{2}dx
- \frac{\ell_1}{2}\alpha_1 \left| \partial _{x}u_{1,n}(0)\right| ^{2}+\frac{1}{2}\alpha_1\int_{0}^{\ell _{1}}\left| \partial _{x}u_{1,n}\right|
^{2}dx && \notag\\
-\beta_1 Re\left\langle \partial _{x}\theta_{1,n},p_1\partial _{x}u_{1,n}\right\rangle
-\ell_1Re\left( \textbf{i}w_np_1f_{1,n}(0)\overline{u_{2,n}}(0)\right)=\circ (1),\label{ee1}&&\\
\frac{1}{2}\int_{0}^{\ell _{1}}\left| \theta_{1,n}\right|
^{2}dx-Re\left(  \frac{1}{\textbf{i}w_n}\gamma_{1}\left\langle p_1\partial _{x}\theta_{1,n},\partial _{x}q_{1,n}\right\rangle\right)\notag && \\
+\delta_1 Re\left\langle \partial _{x}\theta_{1,n},p_1\partial _{x}u_{1,n}\right\rangle
=\circ (1) \label{ee2}&&
\end{eqnarray}
and
\begin{eqnarray}
&&-\frac{\ell_1}{2}\tau _1\left| q_{1,n}(0)\right| ^{2}+\frac{1}{2}\tau _1\int_{0}^{\ell _{1}}\left| q_{1,n}\right|
^{2}dx+Re \left( \frac{\kappa_1}{\textbf{i}w_n}\left\langle p_1\partial _{x}\theta_{1,n},\partial _{x}q_{1,n}\right\rangle\right)  
=\circ (1). \label{ee3}
\end{eqnarray}

Multiplying (\ref{ee1}) by $\delta_1$, (\ref{ee2}) by $\beta_1$ and (\ref{ee3}) by $\frac{\beta_1 \gamma_1}{\kappa_1}$, then summing, we get, by taking into account (\ref{ee0}) and that $q_{1,n}=o(1)$,
\begin{eqnarray}
\frac{1}{2} w_{n}^{2}\left| u_{1,n}(0)\right| ^{2}
+ \frac{1}{2}\alpha_1 \left| \partial _{x}u_{1,n}(0)\right| ^{2}&&\\
+Re\left( \textbf{i}w_nf_{1,n}(0)\overline{u_{2,n}}(0)\right)+
\frac{1}{2}\left| \theta_{1,n}(0)\right| ^{2}+
\frac{1}{2}\tau _1\left| q_{1,n}(0)\right| ^{2} 
=\circ (1).&& \label{ee4}
\end{eqnarray}
Using that $f_n(0)\longrightarrow 0$, we get from (\ref{ee4}),
\begin{equation}
w_{n}u_{1,n}(0),\;\partial _{x}u_{1,n}(0),\;\theta
_{1,n}(0),\;q_{1,n}(0)\longrightarrow 0.  \label{j1}
\end{equation}

Similarly, Taking the inner product of (\ref{e35}) by $(\ell_2-x)\partial _{x}u_{2,n}$ in $L^2(0,\ell_2),$  we obtain
\begin{eqnarray}
\frac{\ell_1}{2}w_{n}^{2}\left| u_{2,n}(0)\right|
^{2}-\frac{1}{2}\int_{0}^{\ell _{2}}w_{n}^{2}\left| u_{2,n}\right|
^{2}dx-\ell_2Re\left(
\alpha _{2}\partial _{x}^{3}u_{2,n}(0)\partial _{x}\overline{u_{2,n}}(0)\right)+ \frac{\ell_2}{2}\alpha_2\left| \partial _{x}^2u_{2,n}(0)\right| ^{2} \notag \\ 
-\frac{3}{2}\alpha_1\int_{0}^{\ell _{2}}\left| \partial _{x}^2u_{2,n}\right|
^{2}dx+\ell_2Re\left(\textbf{i}w_nf_{2,n}(0)\overline{u_{2,n}}(0)\right)=\circ (1),\label{ee12}
\end{eqnarray}

By taking into account (\ref{j1}) and transmission conditions at $x=0$,
\begin{eqnarray}
 \frac{1}{2}\alpha_2 \left| \partial _{x}^2u_{2,n}(0)\right| ^{2}-\frac{1}{2}\int_{0}^{\ell _{2}}w_{n}^{2}\left| u_{2,n}\right|
^{2}dx
-\frac{3}{2}\alpha_1\int_{0}^{\ell _{2}}\left| \partial _{x}^2u_{2,n}\right|
^{2}dx=\circ (1),\label{ee42}
\end{eqnarray}

Next we want prove that the two last terms in the left hand side of (\ref{ee42}) converge to zero as $n$ goes to infinite.

Let $a$ a real number that will be fixed later. The inner product of (\ref{e35}) with  $\frac{1}{%
w_{n}^{1/2}}e^{-aw_{n}^{1/2}x}$ gives
 
\begin{eqnarray*}
&&(\alpha _{2}a^{4}-1)w_{n}^{3/2}\left\langle
u_{2,n},e^{-aw_{n}^{1/2}x}\right\rangle+\frac{1}{w_{n}^{1/2}}\left[ \alpha _{2}
\partial _{x}^{3}u_{2,n} e^{-aw_{n}^{1/2}x}\right] _{0}^{\ell _{2}}\\
&&+\alpha
_{2}a\left[ e^{-aw_{n}^{1/2}x}\partial _{x}^{2}u_{2,n}\right] _{0}^{\ell
_{2}} 
+\alpha _{2}a^{2}w_{n}^{1/2}\left[ e^{-aw_{n}^{1/2}x}\partial _{x}u_{2,n}%
\right] _{0}^{\ell _{2}}+\alpha _{2}a^{3}w_{n}\left[
e^{-aw_{n}^{1/2}x}u_{2,n}\right] _{0}^{\ell _{2}}\\ 
&&  =o(1).
\end{eqnarray*}

We choose $a=\alpha _{2}^{-1/4}$, to obtain 
\begin{eqnarray*}
&&\frac{1}{w_{n}^{1/2}}\left[ \alpha _{2}
\partial _{x}^{3}u_{2,n} e^{-aw_{n}^{1/2}x}\right] _{0}^{\ell _{2}}+\alpha
_{2}a\left[ e^{-aw_{n}^{1/2}x}\partial _{x}^{2}u_{2,n}\right] _{0}^{\ell
_{2}}+\alpha _{2}a^{2}w_{n}^{1/2}\left[ e^{-aw_{n}^{1/2}x}\partial _{x}u_{2,n}%
\right] _{0}^{\ell _{2}}\\
&& 
+\alpha _{2}a^{3}w_{n}\left[
e^{-aw_{n}^{1/2}x}u_{2,n}\right] _{0}^{\ell _{2}}=o(1),
\end{eqnarray*}

then
\begin{equation*}
\frac{1}{w_{n}^{1/2}}\partial _{x}^{3}u_{2,n}(0)+\alpha _{2}^{-1/4}\partial
_{x}^{2}u_{2,n}(0)=o(1).
\end{equation*}
Hence, using the second condition in (\ref{s2''}) by taking into account (\ref{j1}), we get 
\begin{equation*}
\partial _{x}^{2}u_{2,n}(0)=o(1),
\end{equation*}
Return back to (\ref{ee42}), we deduce that the two expressions 
\begin{equation*}
\int_{0}^{\ell _{2}}w_{n}^{2}\left| u_{2,n}\right|
^{2}dx\;\;\text{and}\;\;\int_{0}^{\ell _{2}}\left| \partial _{x}^{2}u_{2,n}\right|
^{2}dx,
\end{equation*}
tend to zero as $n$ goes to infinite.  

Then we conclude that $\left\| y_{n}\right\| \rightarrow 0$ which contradict
the fact that $\left\| y_{n}\right\| =1$ and the proof is then complete.
\end{proof}
\begin{remark}
Note that if the heat damping is present on both the two components then one can check that the coupled system is exponentially stable.
\end{remark}

\section{String/Beam Cattaneo}

In this section we consider the case of a purely elastic string coupled with a thermoelastic beam. We prove that the energy of the whole system is $\frac{1}{t}$-polynomially stable and not $\frac{1}{t^2}$-polynomially stable.
\subsection{Well-posedness and strong stability}
\noindent Define the Hilbert space $\mathcal{H}_{2}$ 
\begin{equation*}
\mathcal{H}_{2}=V\times L^{2}(\mathcal{G})\times L^{2}(0,\ell _{2})\times
L^{2}(0,\ell _{2})
\end{equation*}
with 
norm given by 
\begin{equation*}
\left\| y\right\| _{\mathcal{H}_2}:=\left\| \partial _{x}f_{1}\right\|
^{2}+\left\| g_{1}\right\| ^{2}+\frac{\delta _{2}}{\beta _{2}}\left( \alpha
_{2}\left\| \partial _{x}^{2}f_{2}\right\| ^{2}+\left\| g_{2}\right\|
^{2}\right) +\left\| h_{2}\right\| ^{2}+\frac{\gamma _{2}\tau _{2}}{\kappa
_{2}}\left\| d_{2}\right\| ^{2},
\end{equation*}
where $z=\left( f,g,h_{2},d_{2}\right) .$

Now define the operator $\mathcal{A}_{2}$ on $\mathcal{H}_{2}$ by 
\begin{equation*}
\mathcal{D}(\mathcal{A}_{2})=\left\{ 
\begin{array}{c}
y=(u,v,\theta _{2},q_{2})\in \mathcal{H}_2 \mid v\in V,\;\theta_2\in H^1_0(0,\ell_2),\;q\in H^1(0,\ell_2),\\-\alpha _{2}\partial _{x}^{2}u_{2}+\beta _{2}\theta _{2}\in H^{2}(0,\ell _{2})
\;\text{ and }\;y\text{ satisfies \;(\ref{sbc2})}
\end{array}
\right\}
\end{equation*}
where 
\begin{equation}
\left\{ 
\begin{tabular}{l}
$\partial _{x}^{2}u_{2}(\ell _{2})=0,$ \\ 
$\frac{\delta _{2}}{\beta _{2}}\partial_x\left( \alpha _{2}\partial
_{x}^{2}u_{2}-\beta _{2}\theta _{2}\right)(0) =\alpha
_{1}\partial _{x}u_{1}(0)$
\end{tabular}
\right.  \label{sbc2}
\end{equation}
and

\begin{equation*}
\mathcal{A}_{2}\left( 
\begin{array}{c}
u_{1} \\ 
u_{2} \\ 
v_{1} \\ 
v_{2} \\ 
\theta _{2} \\ 
q_{2}
\end{array}
\right) =\left( 
\begin{array}{c}
v_{1} \\ 
v_{2} \\ 
\alpha _{1}\partial _{x}^{2}u_{1} \\ 
\partial_x^2\left(-\alpha _{2}\partial _{x}^{2}u_{2}+\beta _{2}\theta _{2}\right) \\ 
-\delta _{2}\partial _{xx}v_{2}-\gamma _{2}\partial _{x}q_{2} \\ 
-\frac{1}{\tau _{2}}q_{2}-\frac{\kappa _{2}}{\tau _{2}}\partial _{x}\theta
_{2}
\end{array}
\right) .
\end{equation*}

\noindent Then, putting $y=(u,u_{t},\theta_2 ,q_2),$ we write the system (\ref{s-bc})-(\ref{btcii}) into the following first order evolution equation 
\begin{equation}
\left\{ 
\begin{array}{c}
\frac{d}{dt}y=\mathcal{A}_{2}y, \\ 
y(0)=y^{0}
\end{array}
\right.  \label{sbc3}
\end{equation}
on the energy space $\mathcal{H}_{2}$, where $y^{0}=(u^{0},v^{0},\theta_2
^{0},q_2^{0}).$

For any $y=(u,v,\theta_2,q_2)\in \mathcal{D}(\mathcal{A}_2),$ a direct calculation yields \begin{equation*}
Re(\left\langle \mathcal{A}_2y,y\right\rangle _{\mathcal{H}_2})=-\frac{\gamma
_{1}}{\kappa _{2}}\left\| q_{2}\right\| ^{2} \leq 0.
\end{equation*}
Hence, $\mathcal{A}_2$ is dissipative. Moreover, $0\in \rho(\mathcal{A}_2)$ and
\begin{equation}
\mathbf{i}\mathbb{R}\subset \rho (\mathcal{A}_2). \label{2.2'}
\end{equation}
Then, as in the first case, we have the following result 
\begin{theorem}
The operator $\mathcal{A}_{2}$ is the infinitesimal generator of a $\mathcal{%
C}_{0}$-semigroup of contraction $(S_{2}(t))_{}t\geq 0$ which is  strongly stable.
\end{theorem}

\subsection{Polynomial stability}
In the first part of this section we prove that 
\begin{theorem}
The system $(\mathcal{S}_2)$ is polynomially stable.	More precisely,  there exists $c>0$ such that for all $y^0 \in \mathcal{D}(\mathcal{A}_2), $
\begin{equation*}
\left\| S_2(t)y^{0}\right\|_{\mathcal{H}_2} \leq \frac{c}{(1+t)^{1/2}}\left\| y^{0}\right\| _{%
\mathcal{D}(\mathcal{A}_2)},\;\; \forall t>0,
\end{equation*}
\end{theorem}

\begin{proof}
As for exponential stability, there exists a characterization for a $\mathcal{C}_{0}$-semigroup of contraction to be polynomially stable: Theorem 2.4 in \cite{Tom10}. Taking into account that $\mathbf{i}\mathbb{R}\subset \rho (\mathcal{A}_2)$, The polynomial stability follows from the estimate
\begin{equation}
\limsup_{\beta \in \mathbb{R}, \left| \beta \right| \rightarrow +\infty} \frac{1}{\beta^2}\left\| (\mathbf{i}\beta I-A)^{-1}\right\| _{\mathcal{L}(X)}<\infty,  \label{st3}
\end{equation} 

Suppose the estimate (\ref{st3}) is false. Then there exists a sequence $(w_{n})$ of real numbers, with $%
|w_{n}|\longrightarrow +\infty$ (without loss of generality, we suppose that $w_n>0$), and a sequence of vectors $%
(y_{n})=(u_{n},v_{n},\theta _{2,n},q_{2,n})$ in $\mathcal{D}(\mathcal{A}_2)$ with $%
\left\| y_{n}\right\| _{\mathcal{H}_2}=1$, such that 
\begin{equation*}
\left\| w_{n}^{2}(\mathbf{i}w_{n}I-\mathcal{A}_2)y_{n}\right\| _{%
\mathcal{H}_2}\longrightarrow 0
\end{equation*}
 which is equivalent to 
\begin{eqnarray}
w_{n}^{2}\left( \mathbf{i}w_{n}u_{1,n}-v_{1,n}\right)
&=&f_{1,n}\longrightarrow 0,\;\;\;\text{in}\;H^{1}(0,\ell _{1}),  \label{sbc4}
\\
w_{n}^{2}\left( \mathbf{i}w_{n}v_{1,n}-\alpha _{1}\partial
_{x}^{2}u_{1,n}\right) &=&g_{1,n}\longrightarrow 0,\;\;\;\text{in}%
\;L^{2}(0,\ell _{1}),  \label{sbc5}
\end{eqnarray}
and 
\begin{eqnarray}
w_{n}^{2}\left( \mathbf{i}w_{2,n}u_{2,n}-v_{2,n}\right)
&=&f_{2,n}\longrightarrow 0,\;\;\;\text{in}\;H^{2}(0,\ell _{2}),  \label{sbc6}
\\
w_{n}^{2}\left( \mathbf{i}w_{n}v_{2,n}+\partial_x^2\left(\alpha _{2}\partial
_{x}^{2}u_{2,n}-\beta _{2}\theta _{2,n}\right)\right)
&=&g_{2,n}\longrightarrow 0,\;\;\;\text{in}\;L^{2}(0,\ell _{2}),  \label{sbc7}
\\
w_{n}^{2}\left( \mathbf{i}w_{n}\theta _{2,n}+\gamma _{2}\partial
_{x}^{2}v_{2,n}+\delta _{2}\partial _{x}q_{2,n}\right)
&=&h_{2,n}\longrightarrow 0,\;\;\;\text{in}\;L^{2}(0,\ell _{2}),  \label{sbc8}
\\
w_{n}^{2}\left( \mathbf{i}w_{n}\tau _{2}q_{2,n}+q_{2,n}+\kappa
_{2}\partial _{x}\theta _{2,n}\right) &=&d_{2,n}\longrightarrow 0,\;\;\;%
\text{in}\;L^{2}(0,\ell _{2}),  \label{sbc9}
\end{eqnarray}

We will prove that $y_n$ converges to zero in $\mathcal{H}_2,$ which contradict the fact that $%
\left\| y_{n}\right\| _{\mathcal{H}_2}=1.$

First, since 
\begin{equation*}
Re(\left\langle w_{n}^{2}(\mathbf{i}w_{n}-\mathcal{A}_2%
)y_{n},y_{n}\right\rangle _{\mathcal{H}_2})=-\frac{\gamma _{2}}{\kappa _{2}}%
w_{n}^{2}\left\| q_{2,n}\right\| ^{2},
\end{equation*}
we obtain that $w_{n}q_{2,n}$ converges to $0$ in $L^{2}(0,\ell
_{2})$, and then by (\ref{sbc9}), we deduce that $\partial
_{x}\theta _{2,n}$ converges to zero in $%
L^{2}(0,\ell _{2}).$ 

Now, from (\ref{sbc4})-(\ref{sbc9}) we deduce 
\begin{eqnarray}
w_{n}^{2}\left( w_{n}^{2}u_{1,n}+\alpha _{1}\partial
_{x}^{2}u_{1,n}\right) &=&-g_{1,n}-\mathbf{i}w_{n}f_{1,n},  \label{sbc10} \\
w_{n}^{2}\left( -w_{n}^{2}u_{2,n}+\partial_x^2\left(\alpha _{2}\partial
_{x}^{2}u_{2,n}-\beta _{2}\theta _{2,n}\right)\right) &=&g_{2,n}+%
\mathbf{i}w_{n}f_{2,n},  \label{sbc11} \\
w_{n}^{2}\left( \theta _{2,n}+\frac{1}{\mathbf{i}w_{n}}\gamma
_{2}\partial _{x}q_{2,n}+\delta _{2}\partial _{x}^{2}u_{2,n}\right) &=&\frac{%
1}{\mathbf{i}w_{n}}(h_{2,n}+\partial _{x}^{2}f_{2,n})  \label{sbc12}
\end{eqnarray}
and 
\begin{equation*}
w_{n}^{2}\left\| v_{j,n}\right\| ^{2}-w_{n}^{4}\left\| u_{j,n}\right\|
^{2}\longrightarrow 0,\;\;j=1,2.
\end{equation*}

We divide the rest of the proof into three steps:

\textbf{First step}: We prove some lemmas that will be used frequently later:

Let $p_2(x)=ax+b$, with ($a,b\in\mathbb{C}$), taking the real part of the inner product of (\ref{sbc11}), (\ref{sbc12}) and (\ref{sbc9}) by $p_2w_n^{\beta-2}\partial _{x}u_{2,n},$  $p_2w_n^{\beta-2}\partial _{x}\theta_{2,n}$ and $p_2w_n^{\beta-2}\frac{1}{\mathbf{i}w_n}\partial _{x}q_{2,n}$ respectively (in $L^2(0,\ell_2)$), where $\beta$ is a non negative real number such that $\beta \leq 1$,  we obtain
\begin{eqnarray}
-\left. \frac{1}{2} w_{n}^{2+\beta}\left| u_{2,n}(x)\right| ^{2}p_2(x)\right]_0^{\ell_2} +\frac{1}{2}\int_{0}^{\ell _{2}}w_{n}^{2+\beta}\left| u_{2,n}\right|
^{2}\partial_xp_2 dx-\frac{1}{2}\left. w_n^\beta\alpha_2 \left| \partial _{x}^2u_{2,n}(0)\right| ^{2}p_2(x)\right] _0^{\ell_2}\notag \\
+\left.  w_{n}^{\beta}\partial_x u_{2,n}(x)\left(\alpha _{2}\partial
_{x}^{3}u_{2,n}(x)-\beta _{2}\partial_x\theta _{2,n}(x) \right) p_2(x)\right]_0^{\ell_2} +\frac{3}{2}w_n^\beta\alpha_2\int_{0}^{\ell _{2}}\left| \partial _{x}^2u_{2,n}\right|
^{2}\partial_xp_2dx \notag \\ 
 -Re\left( w_n^\beta\beta _{2}\left\langle \theta _{2,n},\partial
_{x}^2u_{2,n}\partial_x p_2 \right\rangle+w_n^\beta\beta _{2}\left\langle \partial
_{x}\theta _{2,n},\partial _{x}^{2}u_{2,n}p_2\right\rangle 
\right)=\circ (1),\label{ee12}\\
\frac{1}{2}w_n^\beta\int_{0}^{\ell _{2}}\left| \theta_{2,n}\right|^{2}\partial_xp_2
dx- Re\left(\textbf{i} w_n^{\beta-1}\gamma_{2}\left\langle p_2\partial _{x}\theta_{2,n},\partial _{x}q_{2,n}\right\rangle\right)  \notag \\
-\delta_2 w_n^\beta Re\left\langle \partial _{x}\theta_{2,n},p_2\partial _{x}^2u_{2,n}\right\rangle
=\circ (1) \label{ee22}
\end{eqnarray}
and
\begin{eqnarray}
&&-\frac{1}{2}\left. w_n^\beta\tau _2 \left| q_{2,n}(x)\right| ^{2}p_2(x)\right] _0^{\ell_2}+ Re\left(\textbf{i} w_n^{\beta-1}\kappa_2\left\langle p_2\partial _{x}\theta_{2,n},\partial _{x}q_{2,n}\right\rangle\right)  
=\circ (1). \label{ee32}
\end{eqnarray}
We have used that $\theta_{2,n}(0)=\theta_{2,n}(\ell)=0$ and that $w_nq_n=o(1)$.

Multiplying (\ref{ee12}) by $\delta_2$, (\ref{ee22}) by $\beta_2$, and (\ref{ee32}) by $\frac{\beta_2 \gamma_2}{\kappa_2}$, then summing, we get, by taking $p_2(x)=\ell_2-x$,

\begin{lemma}\label{lem7}
\begin{eqnarray}
\frac{\ell_2}{2}\delta_2 w_{n}^{2+\beta}\left| u_{2,n}(0)\right| ^{2}+\frac{\ell_2}{2}w_n^\beta\alpha_2 \delta_2 \left| \partial _{x}^2u_{2,n}(0)\right| ^{2}+\frac{\ell_2}{2}\frac{\beta_2 \gamma_2}{\kappa_2}w_n^\beta\tau _2 \left| q_{2,n}(0)\right| ^{2}\notag\\
-\frac{1}{2}\delta_2\int_{0}^{\ell _{2}}w_{n}^{2+\beta}\left| u_{2,n}\right|
^{2}dx -\frac{3}{2}w_n^\beta\alpha_2 \delta_2\int_{0}^{\ell _{2}}\left| \partial _{x}^2u_{2,n}\right|
^{2}dx-\frac{1}{2}\beta_2 w_n^\beta\int_{0}^{\ell _{2}}\left| \theta_{2,n}\right|
^{2}dx \notag \\ 
 +Re \left( w_n^\beta\beta _{2} \delta_2\left\langle \theta _{2,n},\partial
_{x}^2u_{2,n}\right\rangle\right). 
=\circ (1), \label{ee22'}
\end{eqnarray}
\end{lemma}
The second lemma gives $\alpha _{2}\partial _{x}^{3}u_{2,n}(0)-\beta
_{2}\partial _{x}\theta _{2,n}(0)$ in terms of $\partial
_{x}^{2}u_{2,n}(0)$.
\begin{lemma}\label{lem8}
For every non negative $\beta \leq\frac{1}{2}$ there exists a positive number $a$ such that
\begin{equation}
\frac{w_n^\beta}{w_{n}^{1/2}}(\alpha _{2}\partial _{x}^{3}u_{2,n}(0)-\beta
_{2}\partial _{x}\theta _{2,n}(0))+w_n^\beta\alpha _{2}a\partial
_{x}^{2}u_{2,n}(0)+w_n^{1+\beta}(\alpha _{2}+\beta_2 \delta_2)a^{3} u_{2,n}(0)=o(1).\label{eqlem8}
\end{equation}
\end{lemma}

\begin{proof}
Let $a\in \mathbb{R}$. The inner product of $\alpha _{2}\partial _{x}^{4}u_{2,n}-\beta_2\partial _{x}^{2}\theta_{2,n}$ with $\frac{w_n^\beta}{%
w_{n}^{1/2}}e^{-aw_{n}^{1/2}x}$ gives
\begin{eqnarray*}
w_n^\beta\left\langle \alpha _{2}\partial _{x}^{2}\left( \partial _{x}^{2}u_{2,n}-\beta_2\theta_{2,n}\right) ,\frac{1}{w_{n}^{1/2}}%
e^{-aw_{n}^{1/2}x}\right\rangle =w_n^\beta\left( \frac{1}{w_{n}^{1/2}}\left[\left( \alpha _{2}
\partial _{x}^{3}u_{2,n}-\beta_2\partial _{x}^{2}\theta_{2,n}\right) e^{-aw_{n}^{1/2}x}\right] _{0}^{\ell _{2}}\right. \\
+\alpha
_{2}a\left[ e^{-aw_{n}^{1/2}x}\partial _{x}^{2}u_{2,n}\right] _{0}^{\ell
_{2}} 
+\alpha _{2}a^{2}w_{n}^{1/2}\left[ e^{-aw_{n}^{1/2}x}\partial _{x}u_{2,n}%
\right] _{0}^{\ell _{2}}+\alpha _{2}a^{3}w_{n}\left[
e^{-aw_{n}^{1/2}x}u_{2,n}\right] _{0}^{\ell _{2}}\\ 
\left. +\alpha
_{2}a^{4}w_{n}^{3/2}\left\langle
u_{2,n},e^{-aw_{n}^{1/2}x}\right\rangle -\beta
_{2}a\left[ e^{-aw_{n}^{1/2}x}\theta _{2,n}\right] _{0}^{\ell _{2}}-\beta
_{2}a^{2}w_{n}^{1/2}\left\langle \theta
_{2,n},e^{-aw_{n}^{1/2}x}\right\rangle\right). 
\end{eqnarray*}
The inner product of (\ref{sbc12}) by $w_{n}^{\frac{1}{2}+\beta-2}e^{-aw_{n}^{1/2}x}$ gives
\begin{eqnarray*}
&&-w_{n}^{\beta+1/2}\left\langle \theta _{2,n},e^{-aw_{n}^{1/2}x}\right\rangle \\
&=& -\mathbf{i} w_{n}^{\beta-1/2}\gamma _{2}\left\langle \partial
_{x}q_{2,n},e^{-aw_{n}^{1/2}x}\right\rangle +\delta
_{2}w_{n}^{\beta+1/2}\left\langle \partial
_{x}^{2}u_{2,n},e^{-aw_{n}^{1/2}x}\right\rangle -\frac{w_{n}^{\frac{1}{2}+\beta-2}}{\mathbf{i}w_{n}^{1/2}}%
\left\langle h_{2,n}+\partial _{x}^{2}f_{2,n},e^{-aw_{n}^{1/2}x}\right\rangle
\\
&=& -\mathbf{i}w_{n}^{\beta-1/2}\gamma _{2}\left[ e^{-aw_{n}^{1/2}x}q_{2,n}\right]
_{0}^{\ell _{2}}-\mathbf{i}\gamma _{2}w_n^\beta a\left\langle
q_{2,n},e^{-aw_{n}^{1/2}x}\right\rangle +\delta _{2}w_{n}^{1/2+\beta}\left[
e^{-aw_{n}^{1/2}x}\partial_xu_{2,n}\right] _{0}^{\ell _{2}}  \\
&+&\delta _{2}aw_{n}^{1+\beta}\left[
e^{-aw_{n}^{1/2}x}u_{2,n}\right] _{0}^{\ell _{2}}+\delta
_{2}a^{2}w_{n}^{\beta+3/2}\left\langle
u_{2,n},e^{-aw_{n}^{1/2}x}\right\rangle +o(1)
\end{eqnarray*}

It yields that the inner product of (\ref{sbc11}) with $
w_{n}^{-\frac{1}{2}+\beta-2}e^{-aw_{n}^{1/2}x}$ gives 
\begin{eqnarray*}
&&+((\alpha _{2}+\beta
_{2}\delta _{2})a^{4}-1)w_{n}^{3/2+\beta}\left\langle
u_{2,n},e^{-aw_{n}^{1/2}x}\right\rangle+\frac{w_n^\beta}{w_{n}^{1/2}}\left[\left( \alpha _{2}
\partial _{x}^{3}u_{2,n}-\beta_2\partial _{x}^{2}\theta_{2,n}\right) e^{-aw_{n}^{1/2}x}\right] _{0}^{\ell _{2}}\\
&&+w_n^\beta\alpha
_{2}a\left[ e^{-aw_{n}^{1/2}x}\partial _{x}^{2}u_{2,n}\right] _{0}^{\ell
_{2}} 
+w_n^{\beta+1/2}(\alpha _{2}+\beta_2 \delta_2)a^{2}\left[ e^{-aw_{n}^{1/2}x}\partial _{x}u_{2,n}%
\right] _{0}^{\ell _{2}}\\ 
&& +w_n^{1+\beta}(\alpha _{2}+\beta_2 \delta_2)a^{3}\left[
e^{-aw_{n}^{1/2}x}u_{2,n}\right] _{0}^{\ell _{2}}-w_n^\beta\beta
_{2}a\left[ e^{-aw_{n}^{1/2}x}\theta _{2,n}\right] _{0}^{\ell _{2}} =o(1)
\end{eqnarray*}
Then, (\ref{eqlem8}) is immediate by choosing  $a=\frac{1}{(\alpha _{2}+\beta _{2}\delta _{2})^{1/4}}$.
\end{proof}

Using the boundedness of $y_n$ in Lemma \ref{lem7} with $\beta=0$, we deduce that
\begin{lemma}
\begin{equation}
w_nu_{2,n}(0)=O(1)\;\;\;\text{and}\;\;\;\partial _{x}^{2}u_{2,n}(0)=O(1), \label{boun11}
\end{equation}
then, from Lemma \ref{lem8} we have
\begin{equation}
\frac{1}{w_n^{1/2}}\left( \alpha _{2}\partial
_{x}^{3}u_{2,n}(0)-\beta _{2}\partial _{x}\theta _{2,n}(0)\right) =O(1).\label{boun11'}
\end{equation}
\end{lemma}

\begin{lemma}
For every $0\leq\beta\leq 1$ we have
\begin{eqnarray}
w_{n}^{\beta}\left[ \left\| \theta_{2,n}\right\|
^{2}+\delta_{2}\left\| \partial _{x}^{2}u_{2,n}\right\|
^{2} -2\delta_{2}Re\left\langle \partial_{x} u_{2,n},\partial_{x}\theta_{2,n}\right\rangle\right]  =o(1) \label{sbc13''bis}
\end{eqnarray}
\end{lemma}
\begin{proof}
Multiplying (\ref{sbc12}) by $w_{n}^{\beta-2}\theta_{2,n}$ and $\delta_2w_{n}^{\beta-2}\partial
_{x}^{2}u_{2,n}$ respectively, we get
\begin{equation}
w_{n}^{\beta}\left\| \theta_{2,n}\right\|
^{2}+\mathbf{i} w_{n}^{\beta-1}\gamma_2\left\langle q_n,\partial_x\theta _{2,n}\right\rangle +\delta_{2}w_{n}^{\beta}\left\langle \partial^2_{x} u_{2,n},\theta_{2,n}\right\rangle =o(1)\label{sbc13"}
\end{equation}
and
\begin{eqnarray}
&&\delta_{2}^2w_{n}^{\beta}\left\| \partial _{x}^{2}u_{2,n}\right\|
^{2}+\delta_{2}w_{n}^{\beta}\left\langle \theta _{2,n},\partial
_{x}^{2}u_{2,n}\right\rangle -\mathbf{i}w_{n}^{\beta-1}\delta_{2}\gamma _{2}\left\langle \partial_xq_{2,n},\partial _{x}^{2}u_{2,n}\right\rangle =o(1).  \label{sbc14}
\end{eqnarray}
Summing (\ref{sbc13"}) with (\ref{sbc14}), we find
\begin{eqnarray}
w_{n}^{\beta}\left\| \theta_{2,n}\right\|
^{2}+\delta_{2}^2w_{n}^{\beta}\left\| \partial _{x}^{2}u_{2,n}\right\|
^{2} +2\delta_{2}w_{n}^{\beta}Re \left\langle \partial^2_{x} u_{2,n},\theta_{2,n}\right\rangle\notag \\+\mathbf{i}w_{n}^{\beta-1} \gamma_2\left( \left\langle q_n,\partial_x\theta _{2,n}\right\rangle-\mathbf{i}\delta_{2}w_{n}^{\beta-1}\left\langle \partial_xq_{2,n},\partial _{x}^{2}u_{2,n}\right\rangle\right)   =o(1) \label{sbc13}
\end{eqnarray}
It yelds, after integrating by parts
\begin{eqnarray}
w_{n}^{\beta}\left\| \theta_{2,n}\right\|
^{2}+\delta_{2}^2w_{n}^{\beta}\left\| \partial _{x}^{2}u_{2,n}\right\|
^{2} +2\delta_{2}w_{n}^{\beta}Re \left\langle \partial^2_{x} u_{2,n},\theta_{2,n}\right\rangle\notag \\+\mathbf{i}w_{n}^{\beta-1}\gamma_2\left[ \left\langle q_n,\partial_x\theta _{2,n}\right\rangle+\mathbf{i}\delta_{2}q_{2,n}(0)\partial _{x}^{2}\overline{u_{2,n}}(0)+\mathbf{i}\delta_{2}\left\langle q_{2,n},\partial _{x}^{3}u_{2,n}\right\rangle\right]    =o(1).
 \label{sbc13}
\end{eqnarray}

Using that $w_{n}q_{2,n}, \partial
_{x}\theta _{2,n} \longrightarrow 0$ in $L^{2}(0,\ell _{2})$, we have\begin{equation*}
w_{n}\left\langle q_{2,n},\partial
_{x}\theta_{2,n}\right\rangle \rightarrow 0.
\end{equation*}
From (\ref{sbc11}) we deduce that  $\frac{\partial_x^2\left(\alpha _{2}\partial
_{x}^{2}u_{2,n}-\beta _{2}\theta _{2,n}\right)}{w_{n}}$ is bounded. Then, using the Gagliardo-Nirenberg inequality and the boundedness of $\alpha _{2}\partial
_{x}^{2}u_{2,n}-\beta _{2}\theta _{2,n}$, we deduce that 
$\frac{\left\|\partial_x\left(\alpha _{2}\partial
_{x}^{2}u_{2,n}-\beta _{2}\theta _{2,n}\right)\right\| }{w_{n}^{1/2}}$ is bounded.
Then, taking into account that $\partial_{x}\theta _{2,n} \longrightarrow 0$ in $L^{2}(0,\ell _{2})$, it yields,  $\frac{\left\|\alpha_2\partial_{x}^{3}u_{2,n}\right\| }{w_{n}^{1/2}}=O(1)$. 
Hence, $$w_{n}^{1/2}\left\langle q_{2,n},\partial _{x}^{3}u_{2,n}\right\rangle \rightarrow 0.$$ (We have used that $w_nq _{2,n} \longrightarrow 0$).

Now, using again  that $w_nq _{2,n} \longrightarrow 0$, and that $\frac{1}{w_n}\partial_xq_{2,n}=O(1)$ (by the boundedness of $\theta_{2,n}$ and $\partial_{x}^{2}u_{2,n}$), we can deduce, via the Gagliardo-Nirenberg inequality, that
\begin{equation}
q_{2,n}(0)=o(1). \label{q0}
\end{equation}
Using (\ref{q0}) and (\ref{boun11}),  we get
  $$q_{2,n}(0)\partial _{x}^{2}\overline{u_{2,n}}(0)=o(1).$$
Hence (\ref{sbc13''bis}) holds true.
\end{proof}

\textbf{Second step:} We prove that $w_{n}^{1/2}\left\| \partial _{x}^{2}u_{2,n}\right\|
^{2}$, $w_{n}^{1/2}\left\| \theta
_{2,n}\right\|
^{2}$ and $w_{n}^{1/2} \left\| v_{2,n}\right\| ^{2}$ converge to zero.

By Gagliardo-Nirenberg inequality and the boundedness of $w_nu_n$ and $\partial_x^2u_{2,n}$, we have that $w^{1/2}\partial _{x}u _{2,n}$ is bounded. Recall again that $\partial _{x}\theta _{2,n}=o(1)$, then
\begin{equation}
w_n^{1/2}\left\langle \partial_{x} u_{2,n},\partial_{x}\theta_{2,n}\right\rangle=o(1)\label{diff}
\end{equation}
which implies, using (\ref{sbc13''bis}) with $\beta=1/2$,
\begin{equation}
w_{n}^{1/2}\left\| \partial _{x}^{2}u_{2,n}\right\|
^{2}=o(1)\;\;\text{and}\;\;w_{n}^{1/2}\left\| \theta
_{2,n}\right\|
^{2}=o(1).\label{diff"}
\end{equation}

Taking the inner product of (\ref{sbc11}) with $w_{n}^{-1}u_{2,n}$, we get
\begin{eqnarray}
&&w_{n}^{3}\left\| u_{2,n}\right\| ^{2}=\alpha _{2}w_{n}\|\partial _{x}^{2}u_{2,n}\|^2+\beta _{2}w_n\left\langle
\partial_x\theta _{2,n},\partial
_{x}u_{2,n}\right\rangle \notag\\&-&w_{n}\overline{u_{2,n}}(0)\left( \alpha _{2}\partial _{x}^{3}u_{2,n}(0)-\beta
_{2}\partial _{x}\theta _{2,n}(0)\right) +o(1). \label{sbc16} 
\end{eqnarray}

 Using (\ref{boun11}) and (\ref{boun11'}), we have that
\begin{equation}
w_{n}^{1/2}u_{2,n}(0)\left( \alpha _{2}\partial
_{x}^{3}u_{2,n}(0)-\beta _{2}\partial _{x}\theta _{2,n}(0)\right) =O(1).\label{222}
\end{equation}
and a fortiori
\begin{equation}
u_{2,n}(0)\left( \alpha _{2}\partial
_{x}^{3}u_{2,n}(0)-\beta _{2}\partial _{x}\theta _{2,n}(0)\right) =o(1).\label{222'}
\end{equation}  
Using (\ref{diff"}), (\ref{222'}) and (\ref{diff}) in (\ref{sbc16}) yields		
\begin{equation*}
w_{n}\left\|u_{2,n}\right\|
=o(1).\label{boun5'} 
\end{equation*}
Now we take $\beta=0$ in (\ref{ee22'}) to deduce in particular that
$$w_nu_{2,n}(0)=o(1)$$
which implies, keeping in mind   (\ref{boun11'}),
 \begin{equation*}
w_{n}^{1/2}u_{2,n}(0)\left( \alpha _{2}\partial
_{x}^{3}u_{2,n}(0)-\beta _{2}\partial _{x}\theta _{2,n}(0)\right) =o(1).\label{222""}
\end{equation*}
Return back to (\ref{sbc16}), it yields 
\begin{equation}
w_{n}^{5/2}\left\|u_{2,n}\right\|
^{2}=o(1).\label{boun5}
\end{equation}
\textbf{Third step:} We prove that $w_{n}\left\| \partial _{x}^{2}u_{2,n}\right\|
^{2}$, $w_{n}\left\| \theta
_{2,n}\right\|
^{2}$ and $w_{n} \left\| v_{2,n}\right\| ^{2}$ converge to zero.

Using the estimates of the second step in (\ref{ee22'}) with $\beta=1/2$, we get
$$w_n^{1/4}q_n(0)=o(1),\;\;w_n^{1/4}\partial_x^2u_{2,n}(0)=o(1),\;\;w_n^{5/4}u_{2,n}(0)=o(1),$$
and then by (\ref{eqlem8}),
$$\frac{\alpha _{2}\partial
_{x}^{3}u_{2,n}(0)-\beta _{2}\partial _{x}\theta _{2,n}(0)}{w_n^{1/4}}.$$
We deduce in particular that
\begin{equation}
w_{n}u_{2,n}(0)\left( \alpha _{2}\partial
_{x}^{3}u_{2,n}(0)-\beta _{2}\partial _{x}\theta _{2,n}(0)\right) =o(1).\label{sec2}
\end{equation}
Now, by using again the Gagliardo-Nirenberg inequality, we have for $\varepsilon>0$,
\begin{eqnarray}
|w_{n}\left\langle \partial_{x} u_{2,n},\partial
_{x}\theta_{2,n}\right\rangle|&\leq & \frac{\varepsilon}{2}w_n^2\|\partial_{x} u_{2,n}\|^2+\frac{1}{2\varepsilon}\|\partial_{x} \theta_{2,n}\|^2\notag\\
&\leq & C^2\varepsilon\left(w_n^{3/2}\| u_{2,n}\|w_n^{1/2}\|\partial^2_{x} u_{2,n}\|+w_n^{2}\| u_{2,n}\|^2 \right)+ \frac{1}{2\varepsilon}\|\partial_{x} \theta_{2,n}\|^2\notag\\
&\leq & \frac{C^2\varepsilon}{2}\left(w_n^{3}\| u_{2,n}\|^2+w_n\|\partial^2_{x} u_{2,n}\|^2+2w_n^{2}\| u_{2,n}\|^2 \right)+ \frac{1}{2\varepsilon}\|\partial_{x} \theta_{2,n}\|^2\label{sec3}
\end{eqnarray}

Combining (\ref{sbc13''bis}), with $\beta=1$, (\ref{sbc16}), (\ref{sec2}) and (\ref{sec3}), with $\varepsilon$  small enough we obtain that:

\begin{equation}
w_{n}\left\|\theta_{2,n}\right\|
^{2},\;\;w_{n}\left\| \partial _{x}^{2}u_{2,n}\right\|
^{2},\;\;w_{n}^{3}\left\| u_{2,n}\right\|
^{2}\longrightarrow 0.\label{boun5}
\end{equation}
\textbf{Fourth step:} Conclusion.

Using (\ref{ee22'}) with $\beta=1$ we get
\begin{equation}
w_n^{1/2}q_n(0)=o(1),\;\;w_n^{1/2}\partial_x^2u_{2,n}(0)=o(1),\;\;w_n^{3/2}u_{2,n}(0)=o(1) \label{es1}
\end{equation}
and then by (\ref{eqlem8}),
\begin{equation}
\alpha _{2}\partial
_{x}^{3}u_{2,n}(0)-\beta _{2}\partial _{x}\theta _{2,n}(0)=o(1).\label{es2}
\end{equation}
Finally we prove that $\left\| y_{n}\right\| _{\mathcal{H}}\longrightarrow 0:$ 

To this end, taking the real part of the inner product of
(\ref{sbc5}) by $(\ell_1-x)w_n^{-2}\partial _{x}u_{1,n}$  to get
\begin{equation}
\int_{0}^{\pi}\left( \left| \partial _{x}u_{1,n}(x)\right|
^{2}+w_{n}^{2}\left| u_{1,n}(x)\right| ^{2}\right) dx\longrightarrow 0.
\label{eee}
\end{equation}
We have used (\ref{es1}), (\ref{es2}), the continuity condition of $u_{n}$ at $x=0$, the boundary condition of $u_n$ at $x=\ell_1$ and the damping condition
 at $x=0$, in (\ref{sbc2}). In summary, we have $\left\| y_{n}\right\| _{%
\mathcal{H}}\longrightarrow 0.$ Such result contradicts the hypothesis that $%
y_{n}$ has the unit norm.
\end{proof}

\subsection{Lack of exponential stability}

For simplicity, we take, in this section, $\ell _{1}=\ell _{2}=\pi $. 
\begin{theorem}\label{th4.2}

The system ($\mathcal{S}_2$) is not
exponentially stable in the energy space $\mathcal{H}_2$. Moreover, the associated semigroup is polynomially stable at least for the order $1$ i.e.  for  every  $\alpha< 1$  we have :
\begin{equation}
\limsup_{|\lambda|\rightarrow\infty} |\lambda|^{-\alpha}\left\Vert(i\lambda I - \mathcal {A})^{-1}\right\Vert_{\mathcal{L}(\mathcal{H})}=\infty. \label{optim}
\end{equation}
\end{theorem}

\begin{proof}[Proof of Theorem \ref{th4.2}]
To prove (\ref{optim}), it is enough to show that there exist a bounded
sequence $f_{n}$ in $\mathcal{H}_2$ and a sequence $w_{n}$ of real numbers
with $w_{n}\rightarrow \infty$ such that the solution $y_{n}$ of
\begin{equation}
w_n^\alpha\left( \mathbf{i}w_{n}-\mathcal{A}_2\right) y_{n}=f_{n}  \label{lk1}
\end{equation}
satisfies
\begin{equation}
\lim\limits_{n\rightarrow 0}\left\| y_{n}\right\| _{\mathcal{H}_2}=\infty . \label{inf}
\end{equation}
For each $n\in \mathbb{N},$ we choose $f_{n}=(0,0,-\alpha _{1}\sin \frac{%
w_{n}}{\sqrt{\alpha _{1}}}x,0,0,0)$. Then $f_{n}\in \mathcal{H}_2$ and 
is bounded in $\mathcal{H}_2.$ The solution $y_{n}$ of (\ref{lk1})
satisfies
\begin{eqnarray}
w_n^\alpha\left( \mathbf{i}w_{n}u_{1,n}-v_{1,n}\right)  &=&0,\;\;\;\text{in}%
\;H^{1}(0,\ell _{1}),  \label{lk2} \\
w_n^\alpha\left( \mathbf{i}w_{n}v_{1,n}-\alpha _{1}\partial _{x}^{2}u_{1,n}\right)  &=&-\alpha _{1}%
\sin \frac{w_{n}}{\sqrt{\alpha _{1}}},\;\;\;\text{in}%
\;L^{2}(0,\ell _{1})  \label{lk3}
\end{eqnarray}
and
\begin{eqnarray}
w_n^\alpha\left( \mathbf{i}w_{2,n}u_{2,n}-v_{2,n}\right)  &=&0,\;\;\;\text{in}%
\;H^{2}(0,\ell _{2}),  \label{lk4} \\
w_n^\alpha\left( \mathbf{i}w_{n}v_{2,n}+\partial _{x}^{2}\left( \alpha _{2}\partial _{x}^{2}u_{2,n}-\beta
_{2}\theta _{2,n}\right)\right)   &=&0,\;\;\;\text{in}%
\;L^{2}(0,\ell _{2}),  \label{lk5} \\
w_n^\alpha\left( \mathbf{i}w_{n}\theta _{2,n}+\delta _{2}\partial _{x}^{2}v_{2,n}+\gamma
_{2}\partial _{x}q_{2,n}\right)  &=&0,\;\;\;\text{in}%
\;L^{2}(0,\ell _{2}),  \label{lk6}\\
w_n^\alpha\left( \mathbf{i}w_{n}\tau _2q_{2,n}+q_{2,n}+\kappa _{2}\partial _{x}\theta _{2,n}\right)  &=&0,\;\;\;\text{in}%
\;L^{2}(0,\ell _{2}).  \label{lk7}
\end{eqnarray}
Let $\gamma _{n}:=\frac{w_{n}}{\sqrt{\alpha _{1}}}$ and substitute equation (%
\ref{lk2}) into (\ref{lk3}) to obtain the second order equation
\begin{equation}
w_n^\alpha\left( \gamma _{n}^{2}u_{1,n}+\partial _{x}^{2}u_{1,n}\right)=\alpha _{1}\sin \gamma _{n}x ,
\label{st1}
\end{equation}
which can be solved by
\begin{equation*}
u_{1,n}=\frac{1}{w_n^\alpha}\left(  c_{1}\sin (\gamma _{n}x)+(-\frac{x}{2\gamma _{n}}+c_{2})\cos (\gamma
_{n}x)\right) 
\end{equation*}
where the complex numbers $c_{1}$ and $c_{2}$ depend on $n$.

Substitute (\ref{lk4}) into (\ref{lk5}) and (\ref{lk6}) to get
\begin{eqnarray}
-w_{n}^{2}u_{2,n}+\alpha _{2}\partial _{x}^{4}u_{2,n}-\beta _{2}\partial
_{x}^{2}\theta _{2,n} &=&0,  \label{002} \\
\mathbf{i}w_{n}\theta _{2,n}+\mathbf{i}w_{n}\delta _{2}\partial
_{x}^{2}u_{2,n}+\gamma _{2}\partial _{x}q _{2,n} &=&0.  \label{003}
\end{eqnarray}
Notice that, by (\ref{lk7}), we have $\partial_{x}\theta _{2,n}\in H^1(0,\pi)$.

We replace $\partial _xq_{2,n}$ in (\ref{003}) by $-\frac{\kappa_2}{\mathbf{i}w_n\tau _2+1}\partial _x^2\theta _{2,n},$ then multiplying the obtained result by $\mathbf{i}w_n$ to get
\begin{equation}
\mathbf{i}w_n\theta _{2,n}-\frac{\gamma _2 \kappa _2}{\mathbf{i}w_n\tau _2+1}\partial _{x}^{2}\theta_{2,n}+\mathbf{i}w_n\delta _{2}\partial _{x}^{2}u_{2,n} =0,\;\;\;\text{in}%
\;L^{2}(0,\ell _{2}). \label{003'}
\end{equation}

Combining (\ref{002}) and (\ref{003'}), it yields
\begin{equation}
\alpha_{2} \kappa_{2} \gamma _2 \partial _{x}^{6}u_{2,n}-\mathbf{i}w_n(\mathbf{i}w_n\tau _2+1)(\alpha _{2}+\delta _2\beta
_{2})\partial _{x}^{4}u_{2,n}-\gamma _2\kappa _{2}w_{n}^{2}\partial
_{x}^{2}u_{2,n}+\mathbf{i}w_n^3(\mathbf{i}w_n\tau _2+1)u_{2,n}=0  \label{e6}
\end{equation}

which is an ordinary differential equation of order $6$ in $u_{2,n}.$ Its
characteristic equation is 
\begin{equation}
\alpha_{2} \kappa_{2} \gamma _2 z^{6}-\mathbf{i}w_n(\mathbf{i}w_n\tau _2+1)(\alpha _{2}+\delta _2\beta
_{2})z^{4}-\gamma _2\kappa _{2}w_{n}^{2}z^{2}+\mathbf{i}w_n^3(\mathbf{i}w_n\tau _2+1)
=0.  \label{ec0}
\end{equation}
In the sequel we denote, $L_n:=\mathbf{i}w_n\tau _2+1$, $a:=\alpha_{2} \kappa_{2} \gamma _2,$ $b:=\alpha _{2}+\delta _2\beta_{2}$, and $c:=\gamma _2\kappa _{2}$.
\begin{lemma} \label{lema0}
The equation
\begin{equation}
a z^{6}-\mathbf{i}w_nL_nbz^{4}-cw_{n}^{2}z^{2}+\mathbf{i}w_n^3L_n
=0.,   \label{ec00}
\end{equation}
admits six simple solutions, $\pm z_{1},\pm z_{2}$ and $\pm
z_{3}.$ 
\end{lemma}
\begin{proof}
The solutions of (\ref{ec00}) are the square roots of those of the following
equation
\begin{equation}
a X^{3}-\mathbf{i}w_nL_nbX^{2}-cw_{n}^{2}X+\mathbf{i}w_n^3L_n
=0.  \label{ec}
\end{equation}

By taking $s=X-\mathbf{i}\frac{b}{3a}w_nL_n$, that is $X=s+\mathbf{i}\frac{b}{3a}w_nL_n$
the equation (\ref{ec}) becomes
\begin{equation}
s^{3}+ps+q=0  \label{ec2}
\end{equation}
with
\begin{eqnarray*}
p &:=&\frac{w_n^2}{a}\left( \frac{(L_nb)^2}{3a}-c\right)=\frac{w_n^2L_n^2b^2}{3a^2}\left( 1+o(\frac{1}{w_n^{2-\varepsilon}})\right) \text{ and} \\
q &:=&\mathbf{i}\frac{w_n^3L_n}{a}\left( \frac{2L_n^2b^3}{3^3a^2}+1-\frac{bc}{3a}\right)=
\mathbf{i}2\frac{w_n^3L_nb^3}{3^3a^3}\left( 1+o(\frac{1}{w_n^{2-\varepsilon}}) \right)
\end{eqnarray*}

The discriminant of such equation is
\begin{eqnarray*}
\Delta &:=&q^{2}+\frac{4}{27}p^{3}
=-M^2w_n^{6}L_n^4\left(1+o(\frac{1}{w_n^{2-\varepsilon}}) \right)
\end{eqnarray*}
where $\varepsilon$ is a small enough positive number and $M^2:=\frac{4b^3}{3^3a^4}=\frac{2m}{a}$. 

Notice that for $n$ large enough, $\Delta \neq 0$ then, the equation (\ref{ec2}) admits three simple solutions $%
s_{1}=u_{0}+v_{0},$ $s_{2}=\mathbf{j}u_{0}+\mathbf{j}^{2}v_{0}$ and $s_{3}=%
\mathbf{j}^{2}u_{0}+\mathbf{j}v_{0}$ with $u_{0}^3=\frac{1}{2}(-q+\delta
)$,  $v_{0}^3=\frac{1}{2}(-q-\delta)$ where $\delta$ is a square root of $\Delta$, $\mathbf{j}=e^{\mathbf{i} \frac{2 \pi}{3}}$ and such that $u_0v_0=-\frac{p}{3}$.

We take
\begin{equation*}
\delta=\mathbf{i} Mw_n^{3}L_n^2\left(1+o(\frac{1}{w_n^{2-\varepsilon}}) \right)=\mathbf{i} Mw_n^{3}L_n^3\left(\frac{1}{L_n}+o(\frac{1}{w_n^{2-\varepsilon}}) \right).
\end{equation*}
Then
\begin{equation*}
u_0^3=-\mathbf{i}\frac{m}{2}w_n^3L_n^3\left(1-\frac{M}{mL_n}+o(\frac{1}{w_n^{2-\varepsilon}}) \right).
\end{equation*}
and
\begin{equation*}
v_0^3=-\mathbf{i}\frac{m}{2}w_n^3L_n^3\left(1+\frac{M}{mL_n}+o(\frac{1}{w_n^{2-\varepsilon}}) \right),
\end{equation*}
where $m:=\frac{2b^3}{3^3a^3}$. We choose
\begin{equation*}
u_0=\mathbf{i}\frac{b}{3a}w_nL_n\left(1-\frac{M}{3mL_n} +o(\frac{1}{w_n^{2-\varepsilon}}) \right),
\end{equation*}
then
\begin{equation*}
v_0=\mathbf{i}\frac{b}{3a}w_nL_n\left(1+\frac{M}{3mL_n} +o(\frac{1}{w_n^{2-\varepsilon}}) \right).
\end{equation*}
Hence, the solutions of (\ref{ec2}) are
\begin{equation*}
\left\{ 
\begin{tabular}{l}
$s_{1} =\mathbf{i}\frac{2b}{3a}w_nL_n\left(1+o(\frac{1}{w_n^{2-\varepsilon}}) \right),$ \\
$s_{2} =-\mathbf{i}\frac{b}{3a}w_nL_n\left(1+\mathbf{i}\frac{\sqrt{3}M}{3mL_n}+o(\frac{1}{w_n^{2-\varepsilon}}) \right),$ \\
$s_{3} =-\mathbf{i}\frac{b}{3a}w_nL_n\left(1-\mathbf{i}\frac{\sqrt{3}M}{3mL_n}+o(\frac{1}{w_n^{2-\varepsilon}}) \right).$
\end{tabular}
\right.
\end{equation*}
The solutions of (\ref{ec}) are ($x_1$ associated to $s_2$)
\begin{equation*}
\left\{ 
\begin{tabular}{l}
$x_{1} =\frac{bM}{3\sqrt{3}ma}w_n\left(1+o(\frac{1}{w_n^{1-\varepsilon}}) \right) =b^{-1/2}w_n\left(1+o(\frac{1}{w_n^{1-\varepsilon}}) \right),$ \\
$x_{2} =-\frac{bM}{3\sqrt{3}ma}w_n\left(1+o(\frac{1}{w_n^{1-\varepsilon}}) \right) =-b^{-1/2}w_n\left(1+o(\frac{1}{w_n^{1-\varepsilon}}) \right),$ \\
$x_{3} =\mathbf{i}\frac{b}{a}w_nL_n\left(1+o(\frac{1}{w_n^{2-\varepsilon}}) \right)=-\frac{b}{a}\tau_2 w_n^2\left(1-\mathbf{i}\frac{1}{\tau_2 w_n} +o(\frac{1}{w_n^{2-\varepsilon}})\right),$
\end{tabular}
\right.
\end{equation*}
they are simple and different from $0.$ Consequently, the equation (\ref{eee}%
) admits six simple solutions,  $\pm z_{1},\pm z_{2}$ and $\pm z_{3}$, with
\begin{equation}
\left\{ 
\begin{tabular}{l}
$z_{1} =b^{-1/4}\sqrt{w_n}\left(1+o(\frac{1}{w_n^{1-\varepsilon}}) \right)$ \\ 
$z_{2} =\mathbf{i}b^{-1/4}\sqrt{w_n}\left(1+o(\frac{1}{w_n^{1-\varepsilon}}) \right),$ \\ 
$z_{3} =\mathbf{i}\sqrt{\frac{b\tau_2}{a}}w_n\left(1-\mathbf{i}\frac{1}{2\tau_2 w_n} +o(\frac{1}{w_n^{2-\varepsilon}}) \right).$
\end{tabular}
\right.\label{V}
\end{equation}
\end{proof}

Return back to the ODE (\ref{e6}). The general solution of (\ref{e6}) is given by
\begin{equation}
u_{2,n}=\sum_{k=1}^{3}(d_{k}e^{z_kx}+b_{k}e^{-z_kx})\label{exu}
\end{equation}
where $d_{k},b_{k}$ $k=1,2,3$ are some complex numbers dependent on $n$.

Finally, combining (\ref{002}) and (\ref{003'}), it yields
\begin{equation}
\beta \theta_{2,n}=\mathbf{i}w_n\frac{\gamma_2\kappa_2}{\mathbf{i}w_n\tau_2+1}u_{2,n}-\delta_2\beta_2\partial_x^2u_{2,n}+\frac{\alpha_2\gamma_2\kappa_2}{\mathbf{i}w_n(\mathbf{i}w_n\tau_2+1)}\partial_x^4u_{2,n}
\end{equation}
which will allow us to deduce the explicit expression of $\theta_{2,n}$.  Taking into account  (\ref{exu}) and (\ref{ec0}), 
\begin{equation*}
\beta_{2} \theta _{2,n}=\sum_{k=1}^{3}(-\frac{w_n^2}{z_{k}^{2}}+\alpha_{2}
z_{k}^{2})(d_{k}e^{z_kx}+b_{k}e^{-z_{k}x}).
\end{equation*}
We have also
\begin{equation*}
\beta_{2} \partial _{x}\theta _{2,n}=\sum_{k=1}^{3}(-\frac{w_n^2}{z_{k}}%
+\alpha_{2} z_{k}^{3})(d_{k}e^{z_kx}-b_{k}e^{-z_{k}x})
\end{equation*}
and
\begin{equation*}
\beta_{2} \partial _{x}^{2}\theta _{2,n}=\sum_{k=1}^{3}(-w_n^2+\alpha_{2}
z_{k}^{4})(d_{k}e^{z_kx}+b_{k}e^{-z_{k}x}).
\end{equation*}

Now using the transmission and boundary conditions we have:
\begin{lemma} \label{lem 3.4}
There exists a sequence $\gamma _n:=\frac{w_{n}}{\sqrt{\alpha _{1}}}$ of real positive numbers diverging to infinity such that
\begin{equation}
\frac{1}{w_n^{2\alpha}}\left( \left| -\frac{1}{2\gamma _{n}}+\gamma _{n}c_{1}\right| ^{2}+\left| \gamma
_{n}c_{2}\right| ^{2}\right) \rightarrow +\infty. \label{eess}
\end{equation}
\end{lemma}
\begin{proof}
\textbf{Step 1.} We will built two sequences $S_n$ and $T_n$ such that
\begin{equation}
\gamma _{n}c_{1}(S_{n}-T_{n}\gamma _{n}\tan (\gamma _{n}\pi ))=\frac{%
1}{2\gamma _{n}}S_{n}-\frac{\pi }{2}\gamma _{n}T_{n} \label{seq}
\end{equation}

The transmission conditions at $0$ are expressed as follow :
\begin{eqnarray}
\sum_{k=1}^{3}(d_{k}+b_{k}) &=&\frac{1}{w_n^\alpha} c_{2},  \label{n1} \\
\sum_{k=1}^{3}z_{k}(d_{k}-b_{k}) &=&0,  \label{n2} \\
w_n^2\sum_{k=1}^{3}\frac{1}{%
z_{k}}(d_{k}-b_{k}) &=&\frac{\alpha_1}{w_n^\alpha}\left( -\frac{1}{2\gamma _{n}}+\gamma
_{n}c_{1}\right) ,  \label{n3} \\
\sum_{k=1}^{3}(-\frac{w_n^2}{z_{k}^{2}}+\alpha_{2} z_{k}^{2})(d_{k}+b_{k}) &=&0,
\label{n4}
\end{eqnarray}
and the boundary conditions at $\ell _{1}$ and $\ell _{2}$ are
\begin{eqnarray}
\sum_{k=1}^{3}(d_{k}e^{z_{k}\pi }+b_{k}e^{-z_{k}\pi
}) &=&0, \label{01} \\
\sum_{k=1}^{3}z_{k}^{2}(d_{k}e^{z_{k}\pi }+b_{k}e^{-z_{k}\pi }) &=&0, \label{02} \\
\sum_{k=1}^{3}\frac{1}{z_{k}^{2}}(d_{k}e^{z_{k}\pi }+b_{k}e^{-%
z_{k}\pi }) &=&0, \label{03}\\
c_{1}\sin (\gamma _{n}\pi )+(-\frac{\pi }{2\gamma _{n}}+c_{2})\cos (\gamma
_{n}\pi ) &=&0.\label{04}
\end{eqnarray}

Taking $s_{k}:=d_{k}+b_{k}$ and $t_{k}:=d_{k}-b_{k}.$ Then (\ref{n1}), (\ref
{n4}) and (\ref{n2}), (\ref{n3}) can be rewritten as follow
\begin{eqnarray}
s_{1}+s_{2}+s_{3} &=&\frac{1}{w_n^\alpha} c_{2},  \label{m1} \\
(-\frac{w_n^2}{z_{1}^{2}}+\alpha_{2} z_{1}^{2})s_{1}+(-\frac{w_n^2}{z_{2}^{2}}+\alpha_{2}
z_{2}^{2})s_{2}+(-\frac{w_n^2}{z_{3}^{2}}+\alpha_{2} z_{3}^{2})s_{3} &=&0
\label{m2}
\end{eqnarray}
and
\begin{eqnarray}
z_{1}t_{1}+z_{2}t_{2}+z_{3}t_{3} &=&0,  \label{m3} \\
\frac{1}{z_{1}}t_{1}+\frac{1}{z_{2}}t_{2}+\frac{1}{z_{3}}t_{3}
&=&\frac{1}{w_{n}^{2+\alpha}} \left( -\frac{1}{2\gamma _{n}}+\gamma _{n}c_{1}\right) =:-T.  \label{m4}
\end{eqnarray}
It yields
\begin{eqnarray*}
X_{1}t_{1}+X_{2}t_{2} &=&z_{3}^{2}T, \\
Z_{1}s_{1}+Z_{2}s_{2} &=&Z_{3}c_{2}
\end{eqnarray*}
with
\begin{equation}
\left\{ 
\begin{tabular}{l}
$X_{j} :=z_{j}-\frac{z_{3}^{2}}{z_{j}}=\frac{1}{z_{j}}%
(z_{j}^{2}-z_{3}^{2}),\;\;j=1,2,$ \\
$Z_{j} :=\alpha_{2} (z_{j}^{2}-z_{3}^{2})-w_n^2(\frac{1}{z_{j}^{2}}-\frac{1}{%
z_{3}^{2}})=(\alpha_{2} +\frac{w_n^2}{z_{j}^{2}z_{3}^{2}})(z_{j}^{2}-z_{3}^{2}),\;%
\;j=1,2,$ \\
$Z_{3}:=\frac{w_n^2}{z_{3}^{2}}-\alpha_{2} z_{3}^{2}.$
\end{tabular}
\right.\label{VI}
\end{equation}
We deduce that
\begin{eqnarray}
2d_{2} &=&a_{1}s_{1}+a_{2}t_{1}+a_{3}c_{2}+a_{4}T,  \label{e1} \\
2b_{2} &=&a_{1}s_{1}-a_{2}t_{1}+a_{3}c_{2}-a_{4}T  \label{e2}
\end{eqnarray}
with
\begin{equation}
a_{1}:=-\frac{Z_{1}}{Z_{2}},\;a_{2}:=-\frac{X_{1}}{X_{2}},\;a_{3}:=\frac{Z_{3}}{%
Z_{2}}\;\text{and }a_{4}:=\frac{z_{3}^{2}}{X_{2}}. \label{VII}
\end{equation}

Return back to (\ref{m1}) and (\ref{m3}), we obtain
\begin{eqnarray}
2d_{3} &=&A_{1}s_{1}+A_{2}t_{1}+A_{3}c_{2}+A_{4}T,  \label{e3} \\
2b_{3} &=&A_{1}s_{1}-A_{2}t_{1}+A_{3}c_{2}-A_{4}T  \label{e4}
\end{eqnarray}
with
\begin{equation}
\left\{ 
\begin{tabular}{l}
$A_{1} =-1-a_1=\frac{Z_{1}}{Z_{2}}-1=(\alpha +\frac{w_n^2}{z_{1}^{2}z_{2}^{2}})\frac{%
z_{1}^{2}-z_{2}^{2}}{Z_{2}},$ \\
$A_{2} =\frac{z_{2}}{z_{3}}\frac{X_{1}}{X_{2}}-\frac{z_{1}}{z_{3}}=\frac{%
z_{3}(z_{1}^{2}-z_{2}^{2})}{z_{1}(z_{2}^{2}-z_{3}^{2})},$ \\
$A_{3} =1-a_3=1-\frac{Z_{3}}{Z_{2}}=-\frac{w_n^2-\alpha z_{2}^{4}}{z_{2}^{2}Z_{2}},$ \\
$A_{4} =-\frac{z_2}{z_3}a_4=-\frac{z_{2}z_{3}}{X_{2}}=-\frac{z_{2}^{2}z_{3}}{z_{2}^{2}-z_{3}^{2}%
}.$
\end{tabular}
\right.\label{VIII}
\end{equation}

Replacing $b_{2},d_{2},b_{3}$ and $d_{3}$ given by (\ref{e1})-(\ref{e4}) in (%
\ref{01})-(\ref{03}), we obtain
\begin{equation}
f_{j}d_{1}+g_{j}b_{1}+h_{j}c_{2}+l_{j}T=0,\;\;j=0,2,-2  \label{o1}
\end{equation}
with
\begin{equation}
\left\{ 
\begin{tabular}{l}
$f_{j} :=2z_{1}^{j}e^{z_{1}\pi }+z_{2}^{j}(a_{1}+a_{2})e^{%
z_{2}\pi }+z_{3}^{j}(A_{1}+A_{2})e^{z_{3}\pi }$ \\
$\;\;\;\;\;\;\;\;\;+z_{2}^{j}(a_{1}-a_{2})e^{-z_{2}\pi
}+ z_{3}^{j}(A_{1}-A_{2})e^{-z_{3}\pi },$ \\
$\;\;\;\;\;=e^{z_{1}\pi
}(2z_{1}^{j}+e^{-\varepsilon\sqrt{w_n}}o(1)),$\\
$g_{j} :=z_{2}^{j}(a_{1}-a_{2})e^{z_{2}\pi
}+ z_{3}^{j}(A_{1}-A_{2})e^{z_{3}\pi }$  \\
$\;\;\;\;\;\;\;\;\;+2z_{1}^{j}e^{-z_{1}\pi }+z_{2}^{j}(a_{1}+a_{2})e^{-z_{2}\pi }+z_{3}^{j}(A_{1}+A_{2})e^{-z_{3}\pi },$ \\
$h_{j} :=a_{3}z_{2}^{j}e^{z_{2}\pi }+ A_{3}z_{3}^{j}e^{z_{3}\pi } +a_{3}z_{2}^{j}e^{-z_{2}\pi }+A_{3}z_{3}^{j}e^{-z_{3}\pi },$ \\
$l_{j} :=a_{4}z_{2}^{j}e^{z_{2}\pi }+ A_{4}z_{3}^{j}e^{z_{3}\pi } -a_{4}z_{2}^{j}e^{-z_{2}\pi }-A_{4}z_{3}^{j}e^{-z_{3}\pi }.$
\label{bb"}
\end{tabular}
\right.
\end{equation}
Combining (\ref{04}) and (\ref{m4}) with (\ref{o1}), it follows
\begin{equation}
w_{n}^{2}(f_{j}d_{1}+g_{j}b_{1})=\left( -\frac{\pi }{2\gamma _{n}}+c_{1}\tan
(\gamma _{n}\pi )\right) h_{j}+\frac{\alpha_1}{w_n^\alpha}\left( -\frac{1}{2\gamma _{n}}+\gamma
_{n}c_{1}\right)l_{j},\;\;j=0,2.  \label{022}
\end{equation}

After some combinations we conclude, using (\ref{022}), the following
equation,
\begin{equation}
\alpha_1\left( -\frac{1}{2\gamma _{n}}+\gamma _{n}c_{1}\right) S_{n}=w_{n}^{2}\left( -\frac{\pi
}{2\gamma _{n}}+c_{1}\tan (\gamma _{n}\pi )\right) T_{n}  \label{eqff}
\end{equation}
where
\begin{eqnarray*}
&&S_{n}
:=(g_{2}f_{0}-g_{0}f_{2})l_{-2}-(g_{2}l_{0}-g_{0}l_{2})f_{-2}+(f_{2}l_{0}-f_{0}l_{2})g_{-2},
\\
&&T_{n}
:=-(g_{2}f_{0}-g_{0}f_{2})h_{-2}+(g_{2}h_{0}-g_{0}h_{2})f_{-2}-(f_{2}h_{0}-f_{0}h_{2})g_{-2}.
\end{eqnarray*}
From (\ref{eqff}) one has
\begin{equation*}
\gamma _{n}c_{1}(S_{n}-T_{n}\gamma _{n}\tan (\gamma _{n}\pi ))=\frac{%
1}{2\gamma _{n}}S_{n}-\frac{\pi }{2}\gamma _{n}T_{n}
\end{equation*}

For $k:=\frac{b\pi}{a\kappa_2}$,  $\theta_2=\frac{1}{\mathbf{i}}z_2\pi$, $\theta_3=\frac{1}{\mathbf{i}}(z_3\pi-k)$ and $P:=z_{1}^{-2}(z_{2}^{2}-z_{3}^{2})+z_{2}^{-2}(z_{3}^{2}-z_{1}^{2})+z_{3}^{-2}(z_{1}^{2}-z_{2}^{2}),$ and by using (\ref{bb"}), we have the following expressions of $S_n$ and $T_n$,

\begin{eqnarray*}
S_{n} &=&P\left( D_{1}e^ke^{\mathbf{i}(\theta_2+\theta_3)}+D_{2}e^{-k}e^{-\mathbf{i}(\theta_2+\theta_3)}+D_{3}e^{-k}e^{\mathbf{i}(\theta_2-\theta_3)}+D_{4}e^ke^{-\mathbf{i}(\theta_2-\theta_3)}+e^{-\varepsilon\sqrt{w_n}}o(1)\right)\\
&=&P\left[\left( (D_1+D_4)e^k+(D_2+D_3)e^{-k}\right) \cos(\theta_2)\cos(\theta_3)\right.\\
&-&\left( (D_1-D_4)e^k+(D_2-D_3)e^{-k}\right)\sin(\theta_2)\sin(\theta_3)\\
&+&\mathbf{i}\left( (D_1-D_4)e^k-(D_2-D_3)e^{-k}\right) \sin(\theta_2)\cos(\theta_3)\\&+&\mathbf{i}\left( (D_1+D_4)e^k-(D_2+D_3)e^{-k}\right)\cos(\theta_2)\sin(\theta_3)\\
&+&\left.e^{-\varepsilon\sqrt{w_n}}o(1)\right].
\end{eqnarray*}
and
\begin{eqnarray*}
T_{n} &=&P\left( D^\prime_{1}e^ke^{\mathbf{i}(\theta_2+\theta_3)}+D^\prime_{2}e^{-k}e^{-\mathbf{i}(\theta_2+\theta_3)}+D^\prime_{3}e^{-k}e^{\mathbf{i}(\theta_2-\theta_3)}+D^\prime_{4}e^ke^{-\mathbf{i}(\theta_2-\theta_3)}+e^{-\varepsilon\sqrt{w_n}}o(1)\right)\\
&=&P\left[\left( (D^\prime_1+D^\prime_4)e^k+(D^\prime_2+D^\prime_3)e^{-k}\right) \cos(\theta_2)\cos(\theta_3)\right.\\
&-&\left( (D^\prime_1-D^\prime_4)e^k+(D^\prime_2-D^\prime_3)e^{-k}\right)\sin(\theta_2)\sin(\theta_3)\\
&+&\mathbf{i}\left( (D^\prime_1-D^\prime_4)e^k-(D^\prime_2-D^\prime_3)e^{-k}\right) \sin(\theta_2)\cos(\theta_3)\\
&+&\mathbf{i}\left( (D^\prime_1+D^\prime_4)e^k-(D^\prime_2+D^\prime_3)e^{-k}\right)\cos(\theta_2)\sin(\theta_3)\\
&+&\left.e^{-\varepsilon\sqrt{w_n}}o(1)\right].
\end{eqnarray*}
where $\varepsilon$ is a small enough positive number and with
\begin{eqnarray*}
&&D_{1} :=bA_4-a_4B,\;\; D_{2} :=a_4A-aA_4,\\
&&D_{3} :=-(bA_4+a_4A),\;\;D_{4} :=aA_4+Ba_4 \\
\end{eqnarray*}
and
\begin{eqnarray*}
&&D^\prime_{1} :=bA_3-a_3B,\;\;D^\prime_{2} :=aA_3-a_3A, \\
&&D^\prime_{3} :=bA_3-a_3A,\;\;D^\prime_{4} :=aA_3-Ba_3. \\
\end{eqnarray*}

\textbf{Step 2.} Estimate of $S_n$ and $T_n$ and conclusion.
We focus now on the asymptotic behavior of $S_{n}$ and $T_{n}$
under $n$. To do this, we need the estimates of $a_i, \;\;i=1,\cdots,4$ and $A_i, \;\;i=1,\cdots,4$. 

It remains first to estimate $Z_1, Z_2,Z_3$ and $X_1, X_2$. Using (\ref{V}) in (\ref{VI}), we get
\begin{align*}
    Z_1 =\frac{b}{a}\alpha _2\tau_2 w_n^2\left(1+o(\frac{1}{w_n^{1-\varepsilon}}) \right), & 
    Z_2 =\frac{b}{a}\alpha _2\tau_2 w_n^2\left(1+o(\frac{1}{w_n^{1-\varepsilon}}) \right), &
    Z_3 =\frac{b}{a}\alpha _2\tau_2 w_n^2\left(1+o(\frac{1}{w_n^{1-\varepsilon}}) \right),
    \end{align*}
    \begin{align*}
    X_1 &=\frac{b^{5/4}}{a}\tau_2 w_n^{3/2}\left(1+o(\frac{1}{w_n^{1-\varepsilon}}) \right),&
    X_2 &=-\mathbf{i} \frac{b^{5/4}}{a}\tau_2 w_n^{3/2}\left(1+o(\frac{1}{w_n^{1-\varepsilon}}) \right).
    \end{align*}

So  by (\ref{VII}) and (\ref{VIII}), we obtain
\begin{align*}
    a_1 &=-\left(1+o(\frac{1}{w_n^{1-\varepsilon}}) \right), & 
    a_2 &=-\mathbf{i} \left(1+o(\frac{1}{w_n^{1-\varepsilon}}) \right),\\
    a_3&=\left(1+o(\frac{1}{w_n^{1-\varepsilon}}) \right),&
    a_4&=-\mathbf{i}\frac{1}{b^{1/4}}\sqrt{w_n} \left(1++o(\frac{1}{w_n^{1-\varepsilon}}) \right)
    \end{align*}
and
\begin{align*}
    A_1 &=\frac{1}{\sqrt{w_n}} o(\frac{1}{w_n^{1/2-\varepsilon}}),&
A_2&=2\mathbf{i}\frac{a^{1/2}}{b^{3/4}\tau_2^{1/2}}\frac{1}{\sqrt{w_n}} \left(1+o(\frac{1}{w_n^{1-\varepsilon}}) \right),\\
A_3 &=o(\frac{1}{w_n^{1-\varepsilon}}),&
A_4&=\mathbf{i}\frac{a^{1/2}}{b\tau_2^{1/2}} \left(1+o(\frac{1}{w_n^{1-\varepsilon}}) \right).
    \end{align*}
    
So, we have the following estimates of $D_1\pm D_4$, $D_2\pm D_3$, $D_1^\prime\pm D_4^\prime$ and $D_2^\prime\pm D_3^\prime$,
\begin{align*}
D_1+D_4 &=-2\mathbf{i}\frac{a^{1/2}}{b\tau_2^{1/2}}\left(1 +o(\frac{1}{w_n^{1/2-\varepsilon}})\right), &
D_1-D_4&=-2\frac{a^{1/2}}{b\tau_2^{1/2}}\left(1+o(\frac{1}{w_n^{1/2-\varepsilon}})\right),&\\
D_2+D_3 &=-(D_1+D_4),&
D_2-D_3&=-2\frac{a^{1/2}}{b\tau_2^{1/2}}\left(1+o(\frac{1}{w_n^{1/2-\varepsilon}})\right).
    \end{align*} 

and
\begin{align*}
D^\prime_1+D^\prime_4 &=4\mathbf{i}\frac{a^{1/2}}{b^{3/4}\tau_2^{1/2}}\frac{1}{\sqrt{w_n}}\left(1 +o(\frac{1}{w_n^{1/2-\varepsilon}})\right),&
D^\prime_1-D^\prime_4 &=-\frac{1}{\sqrt{w_n}}o(\frac{1}{w_n^{1/2-\varepsilon}}),&\\
D^\prime_2+D^\prime_3 &=-4\mathbf{i}\frac{a^{1/2}}{b^{3/4}\tau_2^{1/2}}\frac{1}{\sqrt{w_n}}\left(1 +o(\frac{1}{w_n^{1/2-\varepsilon}})\right),&
D^\prime_2-D^\prime_3&=-\left(D^\prime_1-D^\prime_4 \right).
\end{align*}
Hence, taking in mind that $P\neq 0$ (by Vieta's formula associated to equation (\ref{ec})),
\begin{eqnarray*}
\frac{1}{P} S_{n} &=&-4\mathbf{i}\frac{a^{1/2}}{b\tau_2^{1/2}}\left( \sinh(k)\cos\theta_3+\mathbf{i}\cosh(k)\sin\theta_3\right)  \left(1+o(\frac{1}{w_n^{1/2-\varepsilon}})\right)\cos\theta_2\\
&-&4\mathbf{i}\frac{a^{1/2}}{b\tau_2^{1/2}}\left(\sinh(k)\cos\theta_3+\mathbf{i} \cosh(k)\sin\theta_3\right)  \left(1+o(\frac{1}{w_n^{1/2-\varepsilon}})\right)\sin\theta_2+o(\frac{1}{w_n^{1/2-\varepsilon}})\\
&=&4\mathbf{i}\frac{a^{1/2}}{b\tau_2^{1/2}}\left( \sinh(k)\cos\theta_3+\mathbf{i}\cosh(k)\sin\theta_3\right)  \left(1+o(\frac{1}{w_n^{1/2-\varepsilon}})\right)\left( \cos\theta_2+\sin\theta_2\right)+o(\frac{1}{w_n^{1/2-\varepsilon}})
\end{eqnarray*}
and
\begin{eqnarray*}
\frac{1}{P} T
_{n} =&\frac{1}{\sqrt{ w_n}}\left[ 8\mathbf{i}\frac{a^{1/2}}{b^{3/4}\tau_2^{1/2}}\left( \sinh(k)\cos\theta_3+\mathbf{i}\cosh(k)\sin\theta_3\right)  \left(1 +o(\frac{1}{w_n^{1/2-\varepsilon}})\right)\cos\theta_2+ o(\frac{1}{w_n^{1/2-\varepsilon}})\right]. 
\end{eqnarray*}

We take  $s_n:=\frac{1}{\sqrt{b}}w_n=\sqrt{\frac{\alpha_1}{b}}\gamma_n$.
By using the Asymptotic Dirichlet's theorem \cite{Shm80}, there exists $%
(p_{n},q_{n})\in \mathbb{N}^{2}$ such that $p_n,q_n\rightarrow \infty$ and
\begin{equation*}
\left| \left(\frac{\alpha_1}{b} \right)^{1/4} -\frac{p_n}{q_n}\right| <\frac{1}{q_{n}^2}\label{dir1}
\end{equation*}
which implies
\begin{equation*}
\left| \left(\frac{\alpha_1}{b} \right)^{1/4}q_n-p_n\right| <\frac{1}{q_n}. \label{dir2}
\end{equation*}
One has $\sqrt{\frac{b}{a}}w_n=\frac{b}{\sqrt{a}}s_n$ and $\gamma_n=\sqrt{\frac{b}{\alpha_1}}s_n$.

We choose $$\sqrt{\gamma_n}=q_n+\frac{\alpha}{q_n^2},$$
then
$$\sqrt{s_n}=\left(\frac{\alpha_1}{b} \right)^{1/4}q_n+\left(\frac{\alpha_1}{b} \right)^{1/4}\frac{\alpha}{q_n^2}.$$
Recall that
$$\frac{\theta_2}{\pi}=\sqrt{s_n} +o(\frac{1}{s_n^{1/2-\varepsilon}}), $$
then 
\begin{equation*}
    \sin(\theta_2) = o(\frac{1}{q_n^{1-\varepsilon}})\;\;\;\text{and}\;\;\; 
    \cos(\theta_2) = 1+o(\frac{1}{q_n^{1-\varepsilon}}).
    \end{equation*}
    Moreover,
$$\tan(\gamma_n\pi)=\frac{2\alpha}{q_n}\left( 1+o(\frac{1}{q_n^{1-\varepsilon}})\right). $$    
    Hence  $$\frac{1}{P}S_n=\mathbf{i}\frac{a^{1/2}}{b\tau_2^{1/2}}\left( \sinh(k)\cos\theta_3+\mathbf{i}\cosh(k)\sin\theta_3\right)  \left(1+o(\frac{1}{q_n^{1-\varepsilon}})\right) \cos\theta_2+o(\frac{1}{q_n^{1-\varepsilon}})$$ and
    \begin{eqnarray*}
    &&\frac{1}{P}T_{n}\gamma _{n}\tan (\gamma _{n}\pi )\\
    &&=2\alpha\alpha_1^{-1/4}8\mathbf{i}\frac{a^{1/2}}{b^{3/4}\tau_2^{1/2}}\left( \sinh(k)\cos\theta_3+\mathbf{i}\cosh(k)\sin\theta_3\right)  \left(1 +o(\frac{1}{q_n^{1-\varepsilon}})\right)\cos\theta_2+ o(\frac{1}{q_n^{1-\varepsilon}}).
\end{eqnarray*}      
 
 We choose $\alpha$ such that $S_{n}-\alpha_{1}T_{n}\gamma _{n}\tan (\gamma _{n}\pi )=o(\frac{1}{q_n^{1-\varepsilon}})$ i.e. $\alpha=\frac{\alpha_1^{1/4}}{4b^{1/4}}$, 
to conclude that
$$|\gamma_nc_1|\geq C \gamma_n^{1-\varepsilon}$$ for some positive constant $C$.
   
Thus we have
\begin{equation*}
\frac{1}{w_n^{2\alpha}}\left( \left| -\frac{1}{2\gamma _{n}}+\gamma _{n}c_{1}\right| ^{2}+\left| \gamma
_{n}c_{2}\right| ^{2}\right) \rightarrow +\infty.
\end{equation*}

The proof of Lemma \ref{lem 3.4} is thus complete.
\end{proof}

Finally, taking the real part of the inner product of (\ref{st1}) with $(\pi
-x)\partial _{x}u_{1,n}$ to obtain
\begin{eqnarray*}
&&-\frac{\pi w_n^{\alpha}}{2}\left( \left| -\frac{1}{2\gamma _{n}} +\gamma _{n}c_{1}\right| ^{2}+
\left| \gamma _{n}c_{2}\right| ^{2}\right)\\&&=-\frac{w_n^{\alpha}}{2}\left( \gamma
_{n}^{2}\left\| u_{1,n}\right\| ^{2}+\left\| \partial _{x}u_{1,n}\right\|
^{2}\right) +Re(\int_{0}^{\pi }\alpha _{1}\sin (\gamma _{n}x)(\pi -x)\partial _{x}%
\overline{u_{1,n}}dx).
\end{eqnarray*}
Then, By (\ref{eess}), $%
\gamma_{n}^{2}\left\| u_{1,n}\right\| ^{2}+\left\| \partial _{x}u_{1,n}\right\|
^{2}$ diverges to infinity. In conclusion $y_{n}$ is not bounded.
\end{proof}

\bibliographystyle{amsplain}
\bibliography{mabib_19_11_20}

\end{document}